\UseRawInputEncoding
\documentclass[12pt]{article}

\usepackage{dsfont}
\usepackage{amsfonts,amsmath,amsthm}
\usepackage{algorithm}
\usepackage[algo2e]{algorithm2e}
\usepackage{color}
\usepackage{graphicx}
\usepackage{stmaryrd}        % provides \llbracket and \rrbracket

\usepackage[paper=a4paper,dvips,top=2cm,left=2cm,right=2cm,
    foot=1cm,bottom=4cm]{geometry}

\begin{document}

\title{Low Rank Approximation of Dual Complex Matrices}
\author{ Liqun Qi\footnote{%
    Department of Applied Mathematics, The Hong Kong Polytechnic University, Hung Hom,
    Kowloon, Hong Kong;
    Department of Mathematics, School of Science, Hangzhou Dianzi University, Hangzhou 310018 China.
    ({\tt maqilq@polyu.edu.hk}).}
     \and  \
    David M. Alexander\footnote{The Institut de Neurosciences des Syst\`emes (INS, UMR1106), Faculty of Medicine, Aix-Marseille University, Marseille 13005, France. ({\tt david.murray.alexander@gmail.com}).}
    \and  \
    Zhongming Chen\footnote{%
    Department of Mathematics, School of Science, Hangzhou Dianzi University, Hangzhou 310018 China.
    ({\tt zmchen@hdu.edu.cn}).}
    \and  \
    Chen Ling\footnote{%
    Department of Mathematics, School of Science, Hangzhou Dianzi University, Hangzhou 310018 China.
    ({\tt macling@hdu.edu.cn}).}
    \and and \
    Ziyan Luo\footnote{Corresponding author, Department of Mathematics,
  Beijing Jiaotong University, Beijing 100044, China. ({\tt zyluo@bjtu.edu.cn}). This author's work was supported by Beijing Natural Science Foundation (Grant No.  Z190002).}
}
\date{\today}
\maketitle

\begin{abstract}
Dual complex numbers can represent rigid body motion in 2D spaces.  Dual complex matrices are linked with screw theory, and have potential applications in various areas.
In this paper, we study low rank approximation of dual complex matrices. We define $2$-norm for dual complex vectors, and Frobenius norm for dual complex matrices.  These norms are nonnegative dual numbers.   We establish the unitary invariance property of dual complex matrices.   We study eigenvalues of square dual complex matrices, and show that an $n \times n$ dual complex Hermitian matrix has exactly $n$ eigenvalues, which are dual numbers. We present a singular value decomposition (SVD) theorem  for dual complex matrices, define ranks and appreciable ranks for dual complex matrices, and study their properties.  We establish an Eckart-Young like theorem for dual complex matrices, and present an algorithm framework for low rank approximation of dual complex matrices via truncated SVD.  The SVD of dual complex matrices also provides a basic tool for Principal Component Analysis (PCA) via these matrices.  Numerical experiments are reported.  %A possible application to brain science is presented for analyzing plane waves and spiral waves in the cortex correlatively.

\medskip

  \medskip


  %\medskip
  %\textbf{AMS subject classifications.}
\end{abstract}

\renewcommand{\Re}{\mathds{R}}
\newcommand{\rank}{\mathrm{rank}}
\renewcommand{\span}{\mathrm{span}}
\newcommand{\X}{\mathcal{X}}
\newcommand{\A}{\mathcal{A}}
\newcommand{\I}{\mathcal{I}}
\newcommand{\B}{\mathcal{B}}
\newcommand{\C}{\mathcal{C}}
\newcommand{\OO}{\mathcal{O}}
\newcommand{\e}{\mathbf{e}}
\newcommand{\0}{\mathbf{0}}
\newcommand{\dd}{\mathbf{d}}
\newcommand{\ii}{\mathbf{i}}
\newcommand{\jj}{\mathbf{j}}
\newcommand{\kk}{\mathbf{k}}
\newcommand{\va}{\mathbf{a}}
\newcommand{\vb}{\mathbf{b}}
\newcommand{\vc}{\mathbf{c}}
\newcommand{\vg}{\mathbf{g}}
\newcommand{\vr}{\mathbf{r}}
\newcommand{\vt}{\rm{vec}}
\newcommand{\vx}{\mathbf{x}}
\newcommand{\vy}{\mathbf{y}}
\newcommand{\vu}{\mathbf{u}}
\newcommand{\vv}{\mathbf{v}}
\newcommand{\y}{\mathbf{y}}
\newcommand{\vz}{\mathbf{z}}
\newcommand{\T}{\top}

\newtheorem{Thm}{Theorem}[section]
\newtheorem{Def}[Thm]{Definition}
\newtheorem{Ass}[Thm]{Assumption}
\newtheorem{Lem}[Thm]{Lemma}
\newtheorem{Prop}[Thm]{Proposition}
\newtheorem{Cor}[Thm]{Corollary}
\newtheorem{Exa}[Thm]{Example}

\section{Introduction}

In 1873, W.K. Clifford \cite{Cl73} introduced dual numbers, dual complex numbers and dual quaternions.  These become the core knowledge of Clifford algebra or geometric algebra.

While dual quaternions can represent rigid body motions in 3D spaces, the primary application of dual complex numbers is in representing rigid body motions in 2D spaces \cite{MKO14}.
Thus, an $m$-dimensional dual complex vector can represent a set of $m$ rigid body motions in 2D space, and an $m \times n$ dual complex matrix represents a linear transformation from the $n$-dimensional dual complex vector space to the $m$-dimensional dual complex vector space.
Dual complex matrices are also linked with screw geometry or screw theory \cite{Fi98}, and have potential applications in  classical mechanics and robotics, complex representations of the Lorentz group in relativity and electrodynamics, conformal mappings in computer vision, the physics of scattering processes, etc., see \cite{Br20}.  %In this paper, we will discuss a possible application of dual complex matrices in brain science.
One important tool in data analysis is Principal Component Analysis (PCA) \cite{CLMW11}.  The core of PCA is low rank approximation of matrices.     Thus, in this paper, we study low rank approximation of dual complex matrices.

Suppose that all the data points are stacked as column vectors of a matrix $M$, the matrix should (approximately) have low rank: mathematically,
$$M= L_0 + N_0,$$
where $L_0$ has low rank and $N_0$ is a small perturbation matrix.  Classical Principal Component Analysis seeks the best rank-$k$ estimate of $L_0$ by solving
$$\min \{ \| M-L\|_2 : {\rm rank}(L) \le k \}.$$
This problem can be solved via the singular value decomposition (SVD) and enjoys a number of optimality properties.  See \cite{CLMW11}.

Now, $M$ is a dual complex matrix.   Hence, we need a theory of low rank approximation for dual complex matrices, including unitary invariance of dual complex matrices, SVD decomposition of dual complex matrices, the rank theory of dual complex matrices, and an Eckart-Young like theorem for dual complex matrices.   In this paper, we study these issues.

In the next section, we introduce the $2$-norm for dual complex vectors. The $2$-norm of a dual complex vector is a nonnegative dual number.
In Section 3, we define the Frobenius norm for dual complex matrices, and establish the unitary equivalence property of dual complex matrices.

In Section 4, we study eigenvalues of dual complex matrices, in particular, dual complex Hermitian matrix.   We show that an $n \times n$ dual complex Hermitian matrix has exactly $n$ eigenvalues, which are nonnegative dual numbers.   We prove a unitary decomposition theorem for dual complex Hermitian matrices.
A singular value decomposition theorem of dual complex matrices is proved in Section 5.  In Section 6, we define ranks and appreciate ranks for dual complex matrices, and study their properties.  An Eckart-Young like theorem for dual complex matrices is established in Section 7.

An algorithm framework for low rank approximation of dual complex matrices via truncated SVD is presented, and numerical experiments are reported,  in Section 8.

A dual complex number can represent a 2D rigid body motion in a 2D space.   Such a 2D space can be a plane or a surface.   For example, the cortex can be regarded as a 2D space. In 2016, Alexander et al. \cite{ANJZMV16} applied principal component analysis (PCA) on scalar valued phase gradients to analyze plane waves in the cortex.   In 2019,  Alexander et al. \cite{ABSV19} used PCA on complex valued unit phase to analyze spiral waves in the cortex.   The core of PCA is singular value decomposition (SVD) of matrices \cite{CLMW11}.  In these two cases, SVD of complex matrices are used.  However, the plane waves and spiral waves in the cortex are correlated. They should not be analyzed separately.    A possible way to overcome this defect is to combine two kinds of analysis together by using SVD of dual complex matrices.    In Section 8, we show that computationally, this combination is possible.  %We describe a possible application of low rank approximation of dual complex matrices to the phase flows in the cortex in Section 9.

%Some final remarks are made in Section 10.

Throughout the paper, scalars, vectors and matrices are denoted by small letters, bold small letters and capital letters, respectively.

\section{The $2$-Norm of Dual Complex Vectors}

\subsection{Dual Numbers}

Denote $\mathbb R$, $\mathbb C$ and $\mathbb D$ as the set of real numbers, the set of complex numbers, and the set of dual numbers, respectively.   A dual number $q$ may be written as $q = q_{st} + q_\I\epsilon$, where $q_{st}$ and $q_\I$ are real numbers,  and $\epsilon$ is the infinitesimal unit, satisfying $\epsilon^2 = 0$.   We call $q_{st}$ the real part or the standard part of $q$, and $q_\I$ the dual part or the infinitesimal part of $q$.  The infinitesimal unit $\epsilon$ is commutative in multiplication with complex numbers.  The dual numbers form a commutative algebra of dimension two over the reals.    If $q_{st} \not = 0$, we say that $q$ is appreciable, otherwise, we say that $q$ is infinitesimal.

A total order $\le$ for dual numbers was introduced in \cite{QLY21}.  Given two dual numbers $p, q \in \mathbb D$, $p = p_{st} + p_\I\epsilon$, $q = q_{st} + q_\I\epsilon$, where $p_{st}$, $p_\I$, $q_{st}$ and $q_\I$ are real numbers, we say that $p \le q$, if either $p_{st} < q_{st}$, or $p_{st} = q_{st}$ and $p_\I \le q_\I$.  In particular, we say that $p$ is positive, nonnegative, nonpositive or negative, if $p > 0$, $p \ge 0$, $p \le 0$ or $p < 0$, respectively.

Suppose that $q = q_{st} + q_\I \epsilon \in \mathbb D$, and $q_{st} > 0$.  Then we define the square root
\begin{equation}
\sqrt{q} = \sqrt{q_{st}} + {q_\I \over 2\sqrt{q_{st}}}\epsilon.
\end{equation}
Conventionally, we have $\sqrt{0} = 0$.

\subsection{Dual Complex Numbers}

Denote the dual complex numbers by  $\mathbb {DC}$.
A dual complex number $q$ has the form \cite{Br20, GT19, Me15}
\begin{equation} \label{ee1}
q = q_{st} + q_\I \epsilon,
\end{equation}
where $q_{st}$ and $q_\I$ are complex numbers.   Again, we call $q_{st}$ the real part or the standard part of $q$, and $q_\I$ the dual part or the infinitesimal part of $q$.  We say that a dual complex number $q$ is appreciable if its standard part is nonzero. Otherwise, we say that it is infinitesimal.
 The multiplication of dual complex numbers is commutative.

 Denote
$$q_{st} = q_1 + q_2\ii, \  \ q_\I = q_3 + q_4\ii,$$
where $q_1, q_2, q_3, q_4 \in \mathbb R$.

Recall that for a complex number $c= a+b\ii$, where $a$ and $b$ are real numbers, its conjugate is
$\bar c = a-b\ii$. We also have $c\bar c = \bar c c = a^2 + b^2 = |c|^2$.

%As in \cite{Bl60, HY03}, we define three different conjugations of a dual complex numbers, according to whether the complex components are conjugated, the dual numbers are conjugated, or both.  We use the notation of \cite{HY03}.
The conjugate of $q = q_{st} + q_\I \epsilon$ is
\begin{equation} \label{ee2}
\bar q = \bar q_{st} + \bar q_\I \epsilon.
\end{equation}
We have
$$q\bar q = \bar q q = |q_{st}|^2 + (q_{st}\bar q_\I + q_\I\bar q_{st})\epsilon = (q_1^2 + q_2^2)+2(q_1q_3+q_2q_4)\epsilon.$$
It is a positive dual number if $q$ is appreciable, or $0$ otherwise.
From this, we may define the magnitude of a dual complex number $q$ as
\begin{equation}
|q| = \left\{\begin{aligned}\sqrt{q_1^2 + q_2^2} + {q_1q_3+q_2q_4 \over \sqrt{q_1^2+q_2^2}}\epsilon, & \ {\rm if}\ q_{st} \not = 0, \\
\sqrt{q_3^2+q_4^2}\epsilon, & \ {\rm otherwise.} \end{aligned}\right.
\end{equation}

By direct calculations, we have the following proposition.

\begin{Prop} \label{p2.1}
The magnitude $|q|$ is a nonnegative dual number for any $q \in \mathbb {DC}$.   If $q$ is appreciable, then
\begin{equation} \label{e8}
|q| = \sqrt{q\bar q}.
\end{equation}
For any $p, q \in \mathbb {DC}$, we have
\begin{itemize}
\item[(i)] $|q| = |\bar q|$;
\item[(ii)] $|q|\ge 0$ for all $q$, and $|q| = 0$ if and only if $q = 0$;
\item[(iii)] $|pq| = |p| |q|$;
\item[(iv)] $|p+q| \le |p|+|q|$.
\end{itemize}

\end{Prop}

%This proposition can be proved by definition.
It is a special case of Theorem 5.1 of \cite{QLY21}.  Hence, we omit its proof here.

%Two dual complex numbers $p$ and $q$ are said to be similar if there exists an appreciable dual complex number $u$ such that $u^{-1}pu = q$.   This is denoted as $p \sim q$.  Then $\sim$ is an equivalence relation.  Denote the equivalence class containing $q$ by $[q]$.    If $p \sim q$, then $|p|=|q|$.   Thus, if $p \sim q$, then either both $p$ and $q$ are appreciable, or both of them are infinitesimal.

\subsection{Dual Complex Vectors}

Denote $\vx = (x_1, \cdots, x_n)^\top \in {\mathbb {DC}}^n$ for dual complex vectors.   We say that $\vx \in {\mathbb {DC}}^n$ is appreciable if at least one of its components is appreciable.  We may also write
$$\vx = \vx_{st} + \vx_\I\epsilon,$$
where $\vx_{st}, \vx_\I \in {\mathbb C}^n$.
Define
$\vx^* := \bar \vx^\top \equiv (\bar x_1, \cdots, \bar x_n)$.   The $2$-norm of $\vx \in {\mathbb {DC}}^n$ is defined as
\begin{equation} \label{2-norm}
\|\vx \|_2 = \left\{\begin{aligned}\sqrt{\sum_{i=1}^n |x_i|^2}, & \ {\rm if}\ \vx_{st} \not = \0, \\
\|\vx_\I\|_2\epsilon, & \ {\rm otherwise.} \end{aligned}\right.
\end{equation}
If $\|\vx\|_2 = 1$, then we say that $\vx$ is a unit dual complex vector.
%If $\vx, \vy \in {\mathbb {DC}}^n$ and $\vx^*\vy = 0$, or $\vx^\top_\epsilon \vy = 0$, then we say that $\vx$ and $\vy$ are orthogonal to each other.
If
$\vx^{(1)}, \cdots, \vx^{(n)} \in {\mathbb {DC}}^n$ and $\left(\vx^{(i)}\right)^*\vx^{(j)} = \delta_{ij}$ for $i, j = 1, \cdots, n$, where $\delta_{ij}$ is the Kronecker symbol, then we say that
$\{ \vx^{(1)}, \cdots, \vx^{(n)} \}$ is an orthonormal basis of ${\mathbb {DC}}^n$.

We have the following proposition.

\begin{Prop}
Suppose that $\vx, \vy \in {\mathbb {DC}}^n$, and $q \in \mathbb {DC}$.   Then,
\begin{itemize}
\item[(i)] $\|\vx\|_2 \ge 0$, and $\|\vx \|_2 = 0$ if and only if $\vx = \0$;
\item[(ii)]$\|q\vx\|_2 = |q|\|\vx\|_2$;
\item[(iii)]$\|\vx + \vy\|_2 \le \|\vx\|_2 + \|\vy\|_2$.
\end{itemize}
\end{Prop}

This proposition can be proved by definition.  It is a special case of Theorem 6.4 of \cite{QLY21}.  Hence, we also omit its proof here.

Note that if both $\vx$ and $q$ are appreciable, then $q\vx$ is appreciable.

%Dual numbers and anti-dual numbers are special cases of dual complex numbers.

\section{Unitary Invariance of Dual Complex Matrices}
%\subsection{Dual Complex Matrices}

The collections of real, complex and dual complex $m \times n$ matrices are denoted by ${\mathbb R}^{m \times n}$, ${\mathbb C}^{m \times n}$ and ${\mathbb {DC}}^{m \times n}$, respectively.

A dual complex matrix $A= (a_{ij}) \in {\mathbb {DC}}^{m \times n}$ can be denoted as
\begin{equation} \label{ee5}
A = A_{st} + A_\I\epsilon,
\end{equation}
where $A_{st}, A_\I \in {\mathbb C}^{m \times n}$.  Again, we call $A_{st}$ and $A_\I$ the standard part and the infinitesimal part of $A$, respectively.    The transpose of $A$ is $A^\top = (a_{ji})$. The conjugate of $A$ is $\bar A = (\bar a_{ij})$.   The conjugate transpose of $A$ is $A^* = (\bar a_{ji}) = \bar A^\top$.   %The dual conjugate transpose of $A$ is $A^\top_\epsilon  = (a_{ji,\epsilon})$.  The dual complex conjugate transpose of $A$ is $A^*_\epsilon  = (\bar a_{ji,\epsilon})$.

Let $A \in {\mathbb {DC}}^{m \times n}$ and $B \in {\mathbb {DC}}^{n \times r}$.   Then we have
\begin{equation}\label{commu} (AB)^\top = B^\top A^\top,~~(AB)^* = B^*A^*.\end{equation}%, etc.

Given a square dual complex matrix $A \in {\mathbb {DC}}^{n \times n}$, it is called invertible (nonsingular) if $AB = BA = I_n$ for some $B \in {\mathbb {DC}}^{n \times n}$, where $I_n$ is the $n\times n$ identity matrix. Such $B$ is unique and denoted by  $A^{-1}$.  %Then we denote $A^{-1} = B$.   It
Matrix $A$ is called Hermitian if $A^* = A$. Write $A = A_{st} + A_\I\epsilon \in {\mathbb {DC}}^{n \times n}$. Then $A$ is Hermitian if and only if both $A_{st}$ and $A_\I$ are complex Hermitian matrices. Matrix $A$ is called unitary if $A^*A= I_n$. Apparently,  $A \in {\mathbb {DC}}^{n \times n}$ is unitary if and only if its column vectors form an orthonormal basis of ${\mathbb {DC}}^n$. %dual Hermitian if $A^\top_\epsilon = A$; dual unitary if $A^\top_\epsilon A = I$; dual complex Hermitian if $A^*_\epsilon = A$; dual complex unitary if $A^*_\epsilon A = I$.
Let $k \leq n$. We say that $A \in \mathbb{DC}^{n\times k}$ is partially unitary if its column vectors are unit vectors and orthogonal to each other.

%Let $A = A_{st} + A_\I\epsilon \in {\mathbb {DC}}^{n \times n}$.

%A square dual complex matrix $A \in {\mathbb {DC}}^{n \times n}$ is unitary if and only if its column vectors form an orthonormal basis of ${\mathbb {DC}}^n$.

Suppose that $A \in {\mathbb {DC}}^{n \times n}$ is invertible.   Then all column and row vectors of $A$ are appreciable.  This can be proved by definition directly.  We also have the following proposition.

\begin{Prop} \label{p2.0}
Suppose that $A \in {\mathbb {DC}}^{n \times n}$ is unitary.   Then $AA^* =I_n$.
\end{Prop}
\begin{proof}   Write $A = A_{st}+A_\I\epsilon$, where $A_{st}, A_\I \in {\mathbb C}^{n \times n}$.   Since $A$ is unitary, $A^*A = I_n$.  This is equivalent to
\begin{equation} \label{aa9}
A_{st}^*A_{st} = I_n
\end{equation}
and
\begin{equation} \label{aa10}
A_{\I}^*A_{st} + A_{st}^*A_\I = O.
\end{equation}
By matrix analysis, from (\ref{aa9}), we have
\begin{equation} \label{aa11}
A_{st}A_{st}^* = I_n
\end{equation}
By (\ref{aa10}), we have
$$A_{st}(A_{\I}^*A_{st} + A_{st}^*A_\I)A_{st}^* = O,$$
i.e.,
\begin{equation} \label{aa12}
A_{st}A_\I^* + A_\I A_{st}^* = O.
\end{equation}
From (\ref{aa11}) and (\ref{aa12}), we have $AA^* =I_n$.
\end{proof}

Dual complex partially unitary matrices have the following properties.

\begin{Prop} \label{pp3.1}
Suppose that $U \in {\mathbb{DC}}^{n \times k}$ is partially unitary, $U = U_{st} + U_\I\epsilon$, where $U_{st}, U_\I \in {\mathbb C}^{n \times k}$, $k \le n$.    Then $U_{st}$ is a complex partially unitary matrix, and
\begin{equation} \label{eq2}
U_{st}^*U_\I + U_\I^*U_{st} = O.%, \ U_{st}U_\I^* +  U_\I U_{st}^* = O.
\end{equation}
\end{Prop}
\begin{proof}
Since $U$ is partially unitary, $U^*U = I_k$, i.e.,
$$\left(U_{st}^* + U_\I^*\epsilon\right)\left(U_{st} + U_\I\epsilon\right) = I_k.$$
%$$\left(U_{st} + U_\I\epsilon\right)\left(U_{st}^* + U_\I^*\epsilon\right) = \red{(I_k, O)}.$$
This implies that
$$U_{st}^*U_{st} + \left(U_{st}^*U_\I + U_\I^*U_{st}\right)\epsilon = I_k.$$
%$$U_{st}U_{st}^* + \left(U_{st}U_\I^* +  U_\I U_{st}^*\right)\epsilon = \red{(I_k, O)}.$$
Then we have (\ref{eq2}), and $U_{st}^*U_{st} = I_k$, i.e., $U_{st}$ is a complex partially unitary matrix.
\end{proof}

\begin{Prop} \label{pp3.1.1}
Suppose that $U = U_{st} + U_\I\epsilon \in {\mathbb{DC}}^{n \times k}$ is partially unitary, where $U_{st}, U_\I \in {\mathbb C}^{n \times k}$, $k < n$.    Then there is a vector $\vv \in {\mathbb {DC}}^n$ such that $(U, \vv) \in {\mathbb{DC}}^{n \times (k+1)}$ is partially unitary.
\end{Prop}
\begin{proof} By Proposition \ref{pp3.1}, $U_{st}$ is a complex partially unitary matrix.  By complex matrix analysis, there is a complex vector $\vv_{st} \in {\mathbb {C}}^n$ such that $(U_{st}, \vv_{st})$ is partially unitary.  Denote the $j$th column vector of $U$ as $\vu_{st, j} + \vu_{\I, j}\epsilon$ for $j=1, \cdots, k$.  Now, let
\begin{equation}  \label{ae.1}
\hat \vv_\I = -\sum_{j=1}^k (\vu_{\I, j}^*\vv_{st})\vu_{st, j}.
\end{equation}
Then for $j = 1, \cdots, k$,
$$(\vv_{st} + \hat \vv_\I\epsilon)^*(\vu_{st, j} + \vu_{\I, j}\epsilon)=0.$$
Then, $\vv_{st} + \hat \vv_\I\epsilon$ is orthogonal to every column vector of $U$.  Note that $\vv_{st} + \hat \vv_\I\epsilon$ is appreciable. %, i.e., $\|\vv_{st} + \hat \vv_\I\epsilon\|_2 \not = 0$.
Let
\begin{equation} \label{ae.2}
\vv = {\vv_{st} + \hat \vv_\I\epsilon \over \|\vv_{st} + \hat \vv_\I\epsilon\|_2}.
\end{equation}
Then we have the desired result.
\end{proof}

Note that the proof of this proposition is constructive.   By complex matrix analysis, we may derive a formula for $\vv_{st}$.  Then we may calculate $\vv$ by (\ref{ae.1}) and (\ref{ae.2}).

\begin{Cor} \label{c3.4}
Suppose that $U \in {\mathbb{DC}}^{n \times k}$ is partially unitary, $k < n$.    Then there is  $V \in {\mathbb {DC}}^{n \times (n-k)}$ such that $(U, V) \in {\mathbb{DC}}^{n \times n}$ is unitary.
\end{Cor}

\begin{Cor} \label{cor3.5} Suppose that $k<n$, $U \in {\mathbb{DC}}^{n \times k}$, $V$, $W \in {\mathbb{DC}}^{n \times (n-k)}$, such that $(U,V)$ and $(U,W)$ are both unitary matrices in ${\mathbb{DC}}^{n \times n}$.  Then there is a unitary matrix $H \in {\mathbb {DC}}^{(n-k) \times (n-k)}$ such that $W = VH$.
\end{Cor}
\begin{proof} %The \red{existence of $V$ and $W$ follows} readily from Corollary \ref{c3.4}. %Denote $$\tilde{H} = (U, V)^*W  = \left[\begin{array}{c}
                         %U^* W \\
                         %V^*W
                       %\end{array}\right]  \in  {\mathbb{DC}}^{n \times (n-k)}.$$
Set $H = V^*W$. Since $(U,V)$ and $(U,W)$ are unitary, we have
$$U^* W =O, ~~VV^* = WW^*,~~V^*V = W^*W = I_{n-k}.$$
Thus,
$$VH = VV^*W = WW^*W = WI_{n-k}=W.$$
%$$W = (U, V)\tilde{H} = (U,V) \left[\begin{array}{c}
                         %O \\
                         %V^*W
                       %\end{array}\right] = VH.$$
Note that $H^*H = W^* V V^* W = W^* WW^*W = I_{n-k}$, and $HH^* = V^*W W^* V = V^* V V^*V = I_{n-k}$. Thus, $H$ is a unitary matrix in ${\mathbb {DC}}^{(n-k) \times (n-k)}$. This completes the proof.
\end{proof}

%For $A, B \in {\mathbb DC}^{m \times n}$, their inner product is defined as
%$$\langle A, B \rangle = {\rm Tr}(A^*B),$$
%where ${\rm Tr}(A^*B)$ denotes the trace of $A^*B$.
%The Frobenius norm of $A$ is
%$$\|A\|_F  = \sqrt{\sum_{i=1}^m \sum_{j=1}^n |a_{ij}|^2}.$$
%$$\|A\|_F = \sqrt{\langle A, A \rangle} = \sqrt{{\rm Tr}(A^*A)} = \sqrt{\sum_{i=1}^m \sum_{j=1}^n |a_{ij}|^2}.$$
%the $\ell_1$-norm of $A = (a_{ij}) \in {\mathbb Q}^{m \times n}$ is defined by $\|A\|_1 = \sum_{i=1}^m \sum_{j=1}^n |a_{ij}|$, and the $\ell_\infty$-norm of $A$ is defined by $\|A\|_\infty = \max_{i, j} |a_{ij}|$ \cite{JNS19}.

Let $A = A_{st} + A_\I \epsilon = (a_{ij}) \in {\mathbb {DC}}^{m \times n}$.  Define the Frobenius norm of $A$ as
\begin{equation}
\|A \|_F = \left\{\begin{aligned}\sqrt{\sum_{i=1}^m \sum_{j=1}^n |a_{ij}|^2}, & \ {\rm if}\ A_{st} \not = O, \\
\|A_\I\|_F\epsilon, & \ {\rm otherwise.} \end{aligned}\right.
\end{equation}

The Frobenius norm of a matrix is actually the $2$-norm of the vectorization of that matrix.  Thus, it has all the properties of a vector norm.

If $A \in {\mathbb {DC}}^{n \times n}$ and $\vx \in {\mathbb {DC}}^n$, then
$$\|A\vx\|_2 \le \|A\|_F \|\vx\|_2.$$
This property can be proved directly.

We also have the following proposition.

\begin{Prop} \label{pp3.2}
Suppose that $U \in {\mathbb{DC}}^{m \times n}$ is partially unitary, and $\vx \in {\mathbb {DC}}^n$.   Then
\begin{equation} \label{eq5}
\|U\vx\|_2 = \|\vx\|_2.
\end{equation}
\end{Prop}
\begin{proof}   Suppose that $\vx = \vx_{st}+\vx_\I \epsilon$ is appreciable.   If $x_i$ is appreciable, then by (\ref{2-norm}), $|x_i|^2 = x_ix_i^*=x_i^*x_i$.   If $x_i$ is infinitesimal, then
$|x_i|^2 = 0 = x_i^*x_i$.   Thus, by (\ref{2-norm}),
$$\|\vx\|_2^2 = \sum_{i=1}^n |x_i|^2 = \vx^*\vx.$$
On the other hand, let $U = U_{st}+U_\I\epsilon$.  Direct calculations lead to $U_{st}^* U_{st}=I_n$.  %By Proposition \ref{pp3.1}, $U_{st}$ is a complex unitary matrix.
Then the standard part of $U\vx$ is $U_{st}\vx_{st} \not = \0$, i.e., $U\vx$ is also appreciable.   We have
$$\|U\vx\|_2^2 = (U\vx)^*(U\vx) = \vx^*U^*U\vx = \vx^*\vx.$$
Hence, $\|U\vx\|_2 = \|\vx\|_2$ in this case.

Now, assume that $\vx$ is infinitesimal.   Then $\vx = \vx_\I\epsilon$, and $U\vx = U_{st}\vx_\I\epsilon$ is also infinitesimal.  We have $\|U_{st}\vx_\I\|_2 = \|\vx_\I\|_2$. Then by (\ref{2-norm}), we still have
$\|U\vx\|_2 = \|\vx\|_2$ in this case.   This proves (\ref{eq5}).
\end{proof}

We are ready to establish unitary invariance of dual complex matrices.
We have the following theorem.

\begin{Thm} \label{tt3.3}
Suppose that $U \in {\mathbb{DC}}^{m \times m}$ and $V \in {\mathbb{DC}}^{n \times n}$ are unitary, and $A \in {\mathbb {DC}}^{m \times n}$.   Then
\begin{equation} \label{eq8}
\|UAV\|_F = \|A\|_F.
\end{equation}
\end{Thm}
\begin{proof}
Write the columns of $A$ as $\va^{(1)}, \cdots, \va^{(n)}$,
$${\rm vec}(A) = \begin{bmatrix} \va^{(1)}\\ \vdots \\ \va^{(n)}\end{bmatrix},$$
and
$${\rm bldg}(U) = {\rm diag}(U, \cdots, U) \in {\mathbb {DC}}^{mn \times mn}.$$
Then ${\rm bldg}(U)$ is an $mn \times mn$ unitary matrix, and ${\rm vec}(A) \in {\mathbb {DC}}^{mn}$.
We have $\|A\|_F = \|{\rm vec}(A)\|_2$ and $\|UA\|_F = \|{\rm bldg}(U){\rm vec}(A)\|_2$.  By Proposition \ref{pp3.2}, we have
$$\|{\rm bldg}(U){\rm vec}(A)\|_2 = \|{\rm vec}(A)\|_2.$$
Thus, $\|UA\|_F = \|A\|_F$.   Similarly, we may show that $\|AV\|_F = \|A\|_F$.  The conclusion follows.
\end{proof}

\section{Eigenvalues of Dual Complex Matrices}

Suppose that $A \in {\mathbb {DC}}^{n \times n}$.   If there are $\lambda \in \mathbb {DC}$ and $\vx \in
{\mathbb {DC}}^n$, where $\vx$ is appreciable, such that
\begin{equation} \label{e2}
A\vx = \lambda\vx,
\end{equation}
then we say that $\lambda$ is an eigenvalue of $A$, with $\vx$ as a corresponding eigenvector.
Here, we request that $\vx$ is appreciable.

%Note that $A\vx = \lambda \vx$ implies $A(q\vx) = \lambda q \vx = (q^{-1}\lambda q)(q \vx)$ if $q$ is appreciable.   Thus, if $\lambda$ is an eigenvalue of $A$, then any member of $[\lambda]$ is an eigenvalue of $A$.

We now establish conditions for a dual complex number to be an eigenvalue of a square dual complex matrix.

%We have the following theorem.

\begin{Thm} \label{t3.2}
Suppose that $A = A_{st}+A_\I\epsilon \in {\mathbb {DC}}^{n \times n}$.   Then $\lambda = \lambda_{st} + \lambda_\I \epsilon$ is an eigenvalue of $A$ with an eigenvector $\vx = \vx_{st}+\vx_I\epsilon$ only if $\lambda_{st}$ is an eigenvalue of the complex matrix $A_{st}$ with an
eigenvector $\vx_{st}$, i.e., $\vx_{st} \not = \0$ and
\begin{equation} \label{e3}
A_{st}\vx_{st} = \lambda_{st} \vx_{st}.
\end{equation}
Furthermore, if $\lambda_{st}$ is an eigenvalue of the complex matrix $A_{st}$ with an
eigenvector $\vx_{st}$, then $\lambda$ is an eigenvalue of $A$ with an eigenvector $\vx$ if and only if $\lambda_\I$ and $\vx_\I$ satisfy
\begin{equation} \label{e4}
\lambda_\I\vx_{st} = A_\I \vx_{st} + A_{st}\vx_\I - \lambda_{st} \vx_\I.
\end{equation}
\end{Thm}
\begin{proof}
By definition, $\lambda$ is an eigenvalue of $A$ with an eigenvector $\vx$ if and only if $\vx_{st} \not = \0$ and $A\vx = \lambda \vx$.   Then $A\vx = \lambda \vx$ is equivalent to
$$(A_{st}+A_\I\epsilon)(\vx_{st}+\vx_I\epsilon) = (\lambda_{st} + \lambda_\I \epsilon)(\vx_{st}+\vx_I\epsilon).$$
This is further equivalent to $A_{st}\vx_{st} = \vx_{st}\lambda_{st}$, i.e., (\ref{e3}), and
\begin{equation} \label{e5}
A_{st}\vx_\I\epsilon + A_\I \vx_{st}\epsilon = \lambda_\I\vx_{st} + \lambda_{st}\vx_\I \epsilon.
\end{equation}
Then (\ref{e5}) is equivalent to
$$A_{st}\vx_\I + A_\I\vx_{st} =  \lambda_\I\vx_{st} +\lambda_{st} \vx_\I,$$
which is further equivalent to (\ref{e4}).    The conclusions of this theorem follow from
these.
\end{proof}

%\section{SVD of Dual Complex Matrices Under Complex Conjugation}

%Under complex conjugation, this is a special case of the results for dual quaternion matrices in \cite{QL21}.     Hence, we only list the results without proof.

Suppose that $A \in {\mathbb {DC}}^{n \times n}$ is a Hermitian matrix.  For any $\vx \in {\mathbb {DC}}^n$, we have
$$(\vx^*A\vx)^* = \vx^*A\vx.$$
This implies that $\vx^*A\vx$ is a dual number.  With the total order of dual numbers defined in Section 2, we may define positive semidefiniteness and positive definiteness of Hermitian matrices in ${\mathbb {DC}}^{n \times n}$.  A Hermitian matrix $A \in {\mathbb {DC}}^{n \times n}$ is called positive semidefinite if for any $\vx \in {\mathbb {DC}}^n$, $\vx^*A\vx \ge 0$; $A$ is called positive definite if for any $\vx \in {\mathbb {DC}}^n$ with $\vx$ being appreciable,  we have $\vx^*A\vx > 0$ and is appreciable.

\begin{Thm} \label{t3.2.1}
An eigenvalue $\lambda$ of a Hermitian matrix $A = A_{st} + A_\I\epsilon \in {\mathbb {DC}}^{n \times n}$ must be a dual number, %hence an eigenvalue of $A$,
and its standard part $\lambda_{st}$ is an eigenvalue of the complex Hermitian matrix $A_{st}$.   Furthermore, assume that $\lambda = \lambda_{st} + \lambda_\I \epsilon$ is an eigenvalue of $A$ with a corresponding eigenvector
$\vx = \vx_{st} + \vx_\I \epsilon \in {\mathbb {DC}}^n$ %is an eigenvector of $A$, associate with the eigenvalue $\lambda$,
where $\vx_{st}, \vx_\I \in {\mathbb C}^n$.   Then we have
\begin{equation} \label{e7.1}
\lambda_\I = {\vx_{st}^* A_\I \vx_{st} \over \vx_{st}^*\vx_{st}}.
\end{equation}

A Hermitian matrix $A \in {\mathbb {DC}}^{n \times n}$ has at most $n$ dual number eigenvalues and no other eigenvalues.

An eigenvalue of a positive semidefinite Hermitian matrix $A \in {\mathbb {DC}}^{n \times n}$ must be a nonnegative dual number.   In that case, $A_{st}$ must be positive semidefinite.
An eigenvalue of a positive definite Hermitian matrix $A \in {\mathbb {DC}}^{n \times n}$ must be an appreciable positive dual number.   In that case, $A_{st}$ must be positive definite.
\end{Thm}
\begin{proof}   Suppose that $A \in {\mathbb {DC}}^{n \times n}$ is a Hermitian matrix, and $\lambda$ is an eigenvalue of $A$, with $\vx$ as the corresponding eigenvector.    Then we have $A\vx = \lambda \vx$, and $\vx$ is appreciable.    We have
$$\vx^*A\vx = \lambda \vx^*\vx.$$
As $\vx = \vx_{st} + \vx_\I \epsilon$, $A = A_{st} + A_\I \epsilon$, and $\lambda = \lambda_{st} + \lambda_\I \epsilon$, considering the infinitesimal part of the equality,  we have
$$\lambda_{st}(\vx_{\I}^* \vx_{st} + \vx_{st}^*\vx_\I ) + \lambda_\I\vx_{st}^*\vx_{st} =
\vx_\I^*A_{st}\vx_{st} + \vx_{st}^*A_{st}\vx_\I + \vx_{st}^*A_\I \vx_{st}.$$
Since $A_{st}\vx_{st} = \lambda_{st}\vx_{st}$, $\vx_{st}^*A_{st} = \lambda_{st}\vx_{st}^*$ and $\lambda_{st}$ is a real number, the above equality reduces to
$$\lambda_\I\vx_{st}^*\vx_{st} = \vx_{st}^*A_\I \vx_{st},$$
which proves (\ref{e7.1}). Hence, $\lambda$ is a dual number and an eigenvalue of $A$.

The other conclusions follow from Theorem \ref{t3.2} and matrix theory.
\end{proof}

Consider eigenvectors of a dual complex Hermitian matrix, associated with two  eigenvalues with distinct standard parts.   We have the following proposition.

\begin{Prop} \label{p3.4.1}
Two eigenvectors of a Hermitian matrix $A \in {\mathbb {DC}}^{n \times n}$, associated with two  eigenvalues with distinct standard parts, are orthogonal to each other.
\end{Prop}
\begin{proof}
Suppose that $\vx$ and $\vy$ are two  eigenvectors of a Hermitian matrix $A \in {\mathbb {DC}}^{n \times n}$, associated with two eigenvalues $\lambda = \lambda_{st} + \lambda_\I \epsilon$ and $\mu = \mu_{st} + \mu_\I \epsilon$, respectively, and $\lambda_{st} \not = \mu_{st}$.   By Theorem \ref{t3.2}, $\lambda$ and $\mu$ are dual numbers.
We have
$$\lambda(\vx^*\vy) = (\lambda\vx)^*\vy = (A\vx)^*\vy = \vx^*A\vy = \vx^*\mu\vy = \mu \vx^*\vy,$$
i.e.,
$$(\lambda - \mu )(\vx^*\vy) = 0.$$
Since $\lambda_{st} \not = \mu_{st}$, $(\lambda - \mu)^{-1}$ exists.   We have $\vx^*\vy = 0$.
\end{proof}

The following is the unitary decomposition theorem of dual complex matrices.

\begin{Thm} \label{t4.1}
Suppose that $A=A_{st}+A_\I \in {\mathbb {DC}}^{n \times n}$ is a Hermitian matrix.  Then there are unitary matrix $U \in {\mathbb {DC}}^{n \times n}$ and a diagonal matrix $\Sigma \in {\mathbb {D}}^{n \times n}$ such that $\Sigma = U^*AU$, where
\begin{equation} \label{eee1.1}
\Sigma\equiv {\rm diag}\left(\lambda_1+\lambda_{1,1}\epsilon,\cdots, \lambda_1+\lambda_{1,k_1}\epsilon, \lambda_2+\lambda_{2,1}\epsilon,\cdots, \lambda_r+\lambda_{r,k_r}\epsilon\right).
\end{equation}
with the diagonal entries of $\Sigma$ being $n$  eigenvalues of $A$,
\begin{equation} \label{eee2.2}
A\vu_{i, j} = (\lambda_i +\lambda_{i, j}\epsilon)\vu_{i, j},
\end{equation}
for $j = 1, \cdots, k_i$ and $i = 1, \cdots, r$, $U = (\vu_{1,1}, \cdots, \vu_{1, k_1}, \cdots, \vu_{r, k_r})$,
$\lambda_1 > \lambda_2 > \cdots > \lambda_r$ are real numbers, $\lambda_i$ is a $k_i$-multiple eigenvalue of $A_{st}$, $\lambda_{i, 1} \ge \lambda_{i, 2} \ge \cdots \ge \lambda_{i, k_i}$ are also real numbers, $\sum_{i=1}^r k_i = n$.   Counting possible multiplicities $\lambda_{i, j}$, the form $\Sigma$ is unique.
\end{Thm}
\begin{proof}   Suppose that $A \in {\mathbb {DC}}^{n \times n}$ is a dual complex Hermitian matrix.  Denote $A = A_{st} + A_\I \epsilon$, where $A_{st}, A_\I \in {\mathbb {C}}^{n \times n}$.  Then $A_{st}$ and $A_\I$ are complex Hermitian matrices.   Thus, there is an $n \times n$ complex unitary matrix $W$ and a real $n \times n$ diagonal matrix $D$ such that $D = WA_{st}W^*$.
Suppose that $D = {\rm diag}(\lambda_1I_{k_1}, \lambda_2I_{k_2}, \cdots, \lambda_rI_{k_r})$, where $\lambda_1 > \lambda_2 > \cdots > \lambda_r$, and $I_{k_i}$ is a $k_i \times k_i$ identity matrix, and $\sum_{i=1}^r k_i = n$.  Let $M = WAW^*$.  Then
\begin{eqnarray*}
&& M \\ & = & D + WA_\I W^*\epsilon\\
& = & \begin{bmatrix}
\lambda_1I_{k_1} +  C_{11}\epsilon &  C_{12}\epsilon & \cdots &   C_{1r} \epsilon\\
  C_{12}^* \epsilon& \lambda_2I_{k_2} +  C_{22}\epsilon & \cdots  &   C_{2r}\epsilon \\
\vdots & \vdots & \ddots & \vdots \\
  C_{1r}^* \epsilon&   C_{2r}^* \epsilon&  \cdots & \lambda_r I_{k_r} +  C_{rr}\epsilon
\end{bmatrix},
\end{eqnarray*}
where each $C_{ij}$ is a complex matrix of adequate dimensions, and each $C_{ii}$ is Hermitian.

Let
$$P = \begin{bmatrix}
I_{k_1} & { C_{12}\epsilon \over \lambda_1-\lambda_2} & \cdots & {  C_{1r}\epsilon \over \lambda_1-\lambda_r}\\
-{  C_{12}^*\epsilon \over \lambda_1-\lambda_2} & I_{k_2} & \cdots  & { C_{2r}\epsilon \over \lambda_2-\lambda_r} \\
\vdots & \vdots & \ddots & \vdots \\
-{ C_{1r}^* \epsilon\over \lambda_1-\lambda_r} & -{  C_{2r}^*\epsilon \over \lambda_2-\lambda_r}  &  \cdots & I_{k_r}
\end{bmatrix}.$$
Then
$$P^* = \begin{bmatrix}
I_{k_1} & -{  C_{12} \epsilon\over \lambda_1-\lambda_2} & \cdots & -{ C_{1r}\epsilon \over \lambda_1-\lambda_r}\\
{ C_{12}^*\epsilon \over \lambda_1-\lambda_2} & I_{k_2} & \cdots  & -{  C_{2r} \epsilon\over \lambda_2-\lambda_r} \\
\vdots & \vdots & \ddots & \vdots \\
{ C_{1r}^* \epsilon\over \lambda_1-\lambda_r} & { C_{2r}^* \epsilon\over \lambda_2-\lambda_r}  &  \cdots & I_{k_r}
\end{bmatrix}.$$
Direct calculations certify $PP^* = P^*P = I_n$ and
$$\Sigma'\equiv PMP^* =(PW)A(PW)^* = {\rm diag}(\lambda_1I_{k_1}+  C_{11} \epsilon, \lambda_2I_{k_2} +  C_{22}\epsilon, \cdots, \lambda_rI_{k_r}+ C_{rr}\epsilon).$$
Since $P$ and $W$ are unitary matrices, then so is $PW$.  %Also, $\Sigma$ is a block diagonal matrix, $\Sigma= {\rm diag}(\lambda_1I_{k_1}+\epsilon C_{11}, \lambda_2I_{k_2} +\epsilon C_{22}, \cdots, \lambda_rI_{k_r}+\epsilon C_{rr}$.  % This is the desired form which we wish to have.
Noting that each $C_{ii}$ is a complex Hermitian matrix, by matrix theory, we can find unitary matrices $U_1\in {\mathbb{C}}^{k_1\times k_1}$, $\cdots$, $U_r\in {\mathbb{C}}^{k_r\times k_r}$ that diagonalize $C_{11}$, $\cdots$, $C_{rr}$, respectively. That is, there exist real numbers $\lambda_{1,1} \ge \cdots \ge \lambda_{1,k_1}$, $\lambda_{2,1} \ge \cdots \ge \lambda_{2, k_2}$, $\cdots$, $\lambda_{r,1} \ge \cdots \ge \lambda_{r,k_r}$ such that
\begin{equation}\label{blocks}
U_{i}^* C_{ii} U_{i} = {\rm diag}\left(\lambda_{i,1}, \cdots, \lambda_{i,k_i}\right), ~~i=1,\cdots, r.
\end{equation}
Denote $V \equiv {\rm diag}\left(U_1, \cdots, U_r\right)$. We can easily verify that $V$ is unitary. Thus, $U\equiv (PW)^*V$ is also unitary.  Denote
$$\Sigma\equiv {\rm diag}\left(\lambda_1+\lambda_{1,1}\epsilon,\cdots, \lambda_1+\lambda_{1,k_1}\epsilon, \lambda_2+\lambda_{2,1}\epsilon,\cdots, \lambda_r+\lambda_{r,k_r}\epsilon\right).$$
Then we have $U^*AU = \Sigma$, as required.  Letting $U = (\vu_{1,1}, \cdots, \vu_{1, k_1}, \cdots, \vu_{r, k_r})$, we have (\ref{eee2.2}). Thus, $\lambda_i+\lambda_{i,j}\epsilon$ are eigenvalues of $A$ with $\vu_{i,j}$ as the corresponding eigenvectors, for $j = 1, \cdots, k_i$ and $i=1, \cdots, r$.   %The uniqueness follows from Theorem \ref{t2.1}.

Note that those $\lambda_{i,j}$'s are all eigenvalues of $C_{ii}$'s, where $C_{ii} = W_iA_\I W_i^*$ with $W_i^*\in {\mathbb{C}}^{n\times k_i}$ the submatrix formed by an orthonormal basis of the eigenspace of $\lambda_{i}$ for $A_{st}$. To show the desired uniqueness of $\Sigma$, it suffices to show that for any other orthonormal basis of the right eigenspace of $\lambda_{i}$ for $A_{st}$, say $\hat{W}_i^*$, those $\lambda_{i,j}$'s are all eigenvalues of $\hat{C}_{ii}=\hat{W}_iA_\I \hat{W}_i^*$ as well. Observe that there exists a unitary matrix $T_i\in {\mathbb{C}}^{k_i\times k_i}$ such that $\hat{W}_i^* = W_i^*T_i$. Then
$$\hat{C}_{ii}  = T_i^* W_i A_\I W_i^* T_i = T_i^* C_{ii} T_i= T_i^*U_{i}{\rm diag}\left(\lambda_{i,1},\cdots, \lambda_{i,k_i}\right)U_{i}^*T_i.$$
Since $U_{i}^*T_i$ is also unitary, we have $\lambda_{i,j}$'s are all eigenvalues of $\hat{C}_{ii}$. The uniqueness is thus proved.
%The last two conclusions follow from Theorem \ref{t2.1}.
\end{proof}

By the above theorem, we have the following theorem.

\begin{Thm} \label{t4.4.1}
Suppose that $A \in {\mathbb {DC}}^{n \times n}$ is Hermitian.  Then $A$ has exactly $n$ eigenvalues, which are all dual numbers.  There are also $n$ eigenvectors, associated with these $n$ eigenvalues.   The Hermitian matrix $A$ is positive semidefinite or definite if and only if all of these eigenvalues are nonnegative, or positive and appreciable, respectively.
\end{Thm}

\section{Singular Value Decomposition of Dual Complex Matrices}

We need the following theorem for SVD of dual complex matrices.

\begin{Thm}\label{special-Herm.1} Suppose that $B \in {\mathbb {DC}}^{m \times n}$ and $A = B^*B$. Then there exists a unitary matrix $U\in {\mathbb {DC}}^{n \times n}$ such that
\begin{equation}\label{block-diag.1}
U^* A U = {\rm diag}(\lambda_1+\lambda_{1,1}\epsilon, \cdots, \lambda_1+\lambda_{1,k_1}\epsilon, \lambda_2+\lambda_{2,1}\epsilon, \cdots, \lambda_s+\lambda_{s,k_s}\epsilon, 0, \cdots, 0),
\end{equation}
where $\lambda_1>\cdots >\lambda_s>0$, $\lambda_{1,1} \ge \cdots \ge \lambda_{1,k_1}$, $\cdots$, $\lambda_{s,1}\ge \cdots \ge \lambda_{s,k_s}$ are real numbers, $\sum_{i=1}^s k_i \le n$.    Counting possible multiplicities $\lambda_{i, j}$, the real numbers $\lambda_i$ and $\lambda_{i, j}$ for $i=1, \cdots, s$ and $j = 1, \cdots, k_i$ are uniquely determined.
\end{Thm}
\begin{proof} By Theorems \ref{t4.1} and \ref{t4.4.1}, since $A$ is a positive semidefinite dual complex Hermitian matrix, $A$ can be diagonalized by $U$ as defined in Theorem \ref{t4.1}, and has exactly $n$ eigenvalues which are all nonnegative dual numbers, and may be denoted as $\lambda_{i}+\lambda_{i,j}\epsilon$, $i=1,\cdots, r$, $j=1,\cdots, k_i$, and $\lambda_1>\cdots >\lambda_r\geq 0$, $\lambda_{1,1} \ge \cdots \ge \lambda_{1,k_1}$, $\cdots$, $\lambda_{r,1}\ge \cdots \ge \lambda_{r,k_r}$. We now need to show that if $\lambda_r =0$, then $\lambda_{r,j}=0$ for every $j=1,\cdots, k_r$.   Note that
\begin{equation}\label{5-0} \vu_{r, j}^* A \vu_{r, j} = \vu_{r, j}^* \lambda_{r,j} \vu_{r, j} \epsilon = \lambda_{r,j} \epsilon, ~~j=1,\cdots, k_r.
\end{equation}
Since
\begin{equation}\label{5-1}
\vu_{r, j}^* A \vu_{r, j}
=  \vu_{r, j}^* B_{st}^*B_{st} \vu_{r, j} + \vu_{r, j}^* \left(B_{st}^*B_{\I}+B_{\I}^*B_{st}\right) \vu_{r, j}\epsilon,
\end{equation}
we have
\begin{equation}\label{5-2}
\vu_{r, j}^* B_{st}^*B_{st} \vu_{r, j} = 0,
\end{equation}
and hence $B_{st} \vu_{r, j} = {\bf 0}$. Therefore, we have
\begin{eqnarray}
&   & \vu_{r, j}^* \left(B_{st}^*B_{\I}+B_{\I}^*B_{st}\right) \vu_{r, j} \nonumber\\
& = &\left(B_{st} \vu_{r, j}\right)^* B_\I \vu_{r, j}+ \vu_{r, j}^*B_{\I}^*\left(B_{st}\vu_{r, j}\right)\nonumber\\
& = & 0. \nonumber
\end{eqnarray}
By  equations \eqref{5-0} and \eqref{5-1}, we have  $\lambda_{r,j}=0$ for every $j=1,\cdots, k_r$. This completes the proof.

\end{proof}

We now have the SVD theorem for dual complex matrices.

\begin{Thm}\label{SVD.1}
Suppose that $B \in {\mathbb {DC}}^{m \times n}$. Then there exists a unitary matrix $\hat{V} \in {\mathbb {DC}}^{m \times m}$ and a unitary matrix $\hat{U} \in {\mathbb {DC}}^{n \times n}$ such that
\begin{equation} \label{e20.1}
\hat{V}^*B\hat{U} = \begin{bmatrix} \Sigma_t & O  \\ O & O \end{bmatrix},
\end{equation}
where $\Sigma_t\in {\mathbb{D}}^{t\times t}$ is a diagonal matrix, taking the form $$\Sigma_t ={\rm diag}\left(\mu_1, \cdots, \mu_r, \cdots, \mu_t \right),$$
$r \le t \le \min \{ m , n \}$, $\mu_1 \ge \mu_2 \ge \cdots \ge \mu_r$ are positive appreciable dual numbers, and $\mu_{r+1} \ge \cdots \ge \mu_t$ are positive infinitesimal dual numbers.    Counting possible multiplicities of the diagonal entries, the form $\Sigma_t$ is unique.
\end{Thm}
\begin{proof}  Let $A = B^* B$. By Theorem \ref{special-Herm.1}, %\blue{Corollary \ref{cor5.2} that}
there exists a complex unitary matrix $U \in {\mathbb {DC}}^{n \times n}$ as defined in Theorem \ref{t4.1} such that $A$ can be diagonalized as in \eqref{block-diag.1}. Let $r = \sum\limits_{j=1}^s k_j$, and
\begin{equation}\label{sigma-cons}
\Sigma_r \equiv \mathrm{diag}\left( \sigma_1+\sigma_{1,1}\epsilon, \cdots, \sigma_1+\sigma_{1,k_1}\epsilon, \sigma_2+\sigma_{2,1}\epsilon,\cdots, \sigma_s+\sigma_{s,k_s}\epsilon
           % \begin{array}{ccccc}
%              \sigma_1+\sigma_{1,1}\epsilon &   &   & &  \\
%              &  \ddots & &  &   \\
%               & &  \sigma_1+\sigma_{1,k_1}\epsilon  &   &  \\
%                &  & & \ddots    &  \\
%                &   &  & &\sigma_s+\sigma_{s,k_s}\epsilon
%            \end{array}
          \right)
\end{equation}
with $\sigma_i = \sqrt{\lambda_{i}}$, $\sigma_{i,j} = { \lambda_{i,j} \over 2\sqrt{\lambda_{i}}}$, $j=1,\cdots, k_i$, $i=1,\cdots, s$. Denote
$U_1 = U_{:,1:r}$ and $U_2 = U_{:, r+1:n}$. By direct calculation, we have
$$AU = [B^* B U_1~~B^*BU_2 ]= [U_1\Sigma_r^2~~O],$$
which leads to
$$ U_1^*B^* B U_1 = \Sigma_r^2, ~~U_2^*B^* B U_2 = O.$$
Therefore, $B U_2 = M \epsilon$ with a complex matrix $M$.

%\blue {The dual quaternion matrix $A^*A$ is positive semi-definite.   By Theorem \ref{t4.3}, there are unitary matrix $V \in {\mathbb {DQ}}^{n \times n}$ and block-diagonal matrix $\Sigma \in {\mathbb {DQ}}^{n \times n}$ such that $\Lambda = (V')^*A^*AV'$, and each block of $\Lambda$ is of the form $\lambda_iI_{k_i} + J_i\epsilon$,
%where $\lambda_i = \sigma_i^2$ is a real nonnegative number, $I_{k_i}$ is the $k_i \times k_i$ identity matrix, $J_i \in {\mathbb Q}^{k_i \times k_i}$ is a Hermitian matrix.  We may further arrange $\sigma_1 \ge \cdots \ge \sigma_r > 0$, $\sigma_{r+1} = \cdots = \sigma_n = 0$.  Let $\sum_{i=1}^r k_i =s$ and $\Lambda_s$ be the left upper $s \times s$ part of $\Lambda$.   Here, $\sigma_i$ are all real nonnegative numbers.      Write
%$V'_1 = (\vv^{(1)}, \cdots, \vv^{(s)})$,  $V'_2 = (\vv^{(s+1)}, \cdots, \vv^{(n)})$, $V' = (V'_1, V'_2)$. Let $\Sigma_s$ be an $s \times s$ block-diagonal matrix such that $\Sigma_s^2 = \Lambda_s$.  Each block of $\Sigma_s$ is of the form $\sigma_iI_{k_i} + {1 \over 2 \sigma_i}J_i\epsilon$.   Then we have
%$$A^*AV'_1 = V_1\Sigma_s^2,$$
%$$(V'_1)^*A^*AV_1 = \Sigma_s^2,$$
%$$A^*AV'_2 = O,$$
%$$(V'_2)^*A^*AV'_2 = O.$$
%Therefore, $AV'_2 = B\epsilon$ for a quaternion matrix $B$.
Let $V_1 = BU_1\Sigma_r^{-1}\in {\mathbb{DC}}^{m\times r}$.
Then
$$V_1^*V_1 = I_s,~~V_1^*BU_1 = V_1^*V_1\Sigma_r = \Sigma_r,~~ V_1^*BU_2 = \left(\Sigma_r^{-1}\right)^*U_1^*(B^*BU_2) =  O.$$
By Corollary \ref{c3.4}, we may take $V_2 \in {\mathbb {DC}}^{m \times (m-r)}$ such that $V = (V_1, V_2)$ is a unitary matrix.  Then
$$V_2^*BU_1= V_2^*V_1\Sigma_r = O,~~V_2^*BU_2 = V_2^*M\epsilon = G\epsilon,$$
where $G$ is an $(m-r) \times (n-r)$ complex matrix.   Thus,
\begin{eqnarray*}
V^*BU & = & \begin{bmatrix} V_1^*BU_1 & V_1^*BU_2 \\ V_2^*BU_1 & V_2^*BU_2 \end{bmatrix}\\
& = & \begin{bmatrix} \Sigma_r & O \\ O & G\epsilon \end{bmatrix}.
\end{eqnarray*}
By matrix theory, there exist complex unitary matrices $W_1\in {\mathbb{C}}^{(m-r)\times (m-r)}$ and $W_2\in{\mathbb{C}}^{(n-r)\times (n-r)}$ such that
$W_1^* G W_2 = D$, where $D\in {\mathbb{R}}^{(m-r)\times (n-r)}$ with $D_{ij}=0$ for any $i\neq j$ and
$$D_{11}\geq \cdots \geq D_{ll}\geq 0$$ with $l= \min\{m-r,n-r\}$. Let $r'$ be the number of positive $D_{ii}$'s (counting multiplicity). Apparently, $D_{11},\cdots, D_{r'r'}$ are the square roots of real positive eigenvalues of the Hermitian matrix $G^*G\in {\mathbb{C}}^{(n-r)\times (n-r)}$.
Denote
\begin{equation}\label{UV_bar}
\hat{V} \equiv V\left(
                  \begin{array}{cc}
                    I_r &   \\
                      & W_1 \\
                  \end{array}
                \right), ~~\mathrm{and}~~\hat{U} \equiv U\left(
                  \begin{array}{cc}
                    I_r &   \\
                      & W_2 \\
                  \end{array}
                \right).
\end{equation}
It is obvious that both $\hat{V}$ and $\hat{U}$ are unitary and $\hat{V}^* B \hat{U} = \begin{bmatrix} \Sigma_r & O  \\ O & D\epsilon \end{bmatrix}$.  Set $t = r+r'$, and
\begin{equation}\label{sigmat}\Sigma_t = {\rm diag}(\Sigma_r, D_{11}\epsilon, \cdots, D_{r'r'}\epsilon).\end{equation} Then we have (\ref{e20.1}).  The uniqueness of $\Sigma_r$ follows from Theorems \ref{t4.1} and \ref{special-Herm.1}.

Now we claim that the real positive numbers $D_{11}, \cdots, D_{r'r'}$ are also unique. Note that $U_1$ and $V_1$ are uniquely determined from Theorems \ref{t4.1} and \ref{special-Herm.1}, while $U_2$ and $V_2$ may have different choices, provided that $U_2$, $V_2$ expand $U_1$ and $V_1$ to get unitary matrices $U$ and $V$, respectively. Among all these choices, pick any two distinct pairs $\{U^{(1)}_2, V^{(1)}_2\}$ and $\{U^{(2)}_2, V^{(2)}_2\}$ (i.e., $U^{(1)}_2\neq  U^{(2)}_2$ or $V^{(1)}_2\neq V^{(2)}_2$, or both).
Applying Corollary \ref{cor3.5}, we can find unitary matrices $H_1\in {\mathbb{DC}}^{(n-r)\times (n-r)}$ and
$H_2\in {\mathbb{DC}}^{(m-r)\times (m-r)}$ such that
\begin{equation}\label{trans}
U^{(2)}_2 = U^{(1)}_2 H_1,~~V^{(2)}_2 = V^{(1)}_2 H_2.
\end{equation}
It is known from Proposition \ref{pp3.1} that
$\left[U_{1,st}, U^{(1)}_{2,st}\right]$ and $\left[U_{1,st}, U^{(2)}_{2,st}\right]$ are complex unitary matrices in ${\mathbb{C}}^{n\times n}$. Similarly,
$\left[V_{1,st}, V^{(1)}_{2,st}\right]$ and $\left[V_{1,st}, V^{(2)}_{2,st}\right]$ are complex unitary matrices in ${\mathbb{C}}^{m\times m}$. Thus, the unitary matrices $H_{1,st}\in {\mathbb{C}}^{(n-r)\times (n-r)}$ and $H_{2,st}\in {\mathbb{C}}^{(m-r)\times (m-r)}$ will satisfy
\begin{equation}
U^{(2)}_{2,st} = U^{(1)}_{2,st}H_{1,st}  ~{\rm ~and~}~V^{(2)}_{2,st} = V^{(1)}_{2,st}H_{2,st}.
\end{equation}
Denote \begin{equation}\label{e31} B U_2^{(1)} = M^{(1)} \epsilon,~~B U_2^{(2)} = M^{(2)} \epsilon,\end{equation}  and
\begin{equation}\label{e32} \left(V_2^{(1)}\right)^*BU_2^{(1)} = \left(V_2^{(1)}\right)^*M^{(1)}\epsilon = G^{(1)}\epsilon,~~\left(V_2^{(2)}\right)^*BU_2^{(2)} = \left(V_2^{(2)}\right)^*M^{(2)}\epsilon = G^{(2)}\epsilon,\end{equation}
where $M^{(i)}$, $G^{(i)}$, $i=1,2$, are complex matrices.
Direct calculations lead to
$$M^{(2)}\epsilon = BU_2^{(2)} = BU_2^{(1)} H_1 = M^{(1)}\epsilon H_1 = M^{(1)}H_{1,st}\epsilon,$$
which implies that $M^{(2)} = M^{(1)}H_{1,st}$. Similarly, we have
\begin{eqnarray}\label{similar}
 G^{(2)}\epsilon &=&\left(V_2^{(2)}\right)^*M^{(2)}\epsilon \nonumber\\
 &= &\left(V_{2,st}^{(2)}\right)^*M^{(2)}\epsilon \nonumber\\
 &= & H_{2,st}^*\left(V_{2,st}^{(1)}\right)^* M^{(1)}H_{1,st} \epsilon \nonumber\\
 & = & H_{2,st}^*G^{(1)}H_{1,st} \epsilon,
\end{eqnarray}
where the last equality follows from the fact that
$$G^{(1)}\epsilon =\left(V_2^{(1)}\right)^*M^{(1)}\epsilon = \left(V_{2,st}^{(1)}\right)^* M^{(1)}\epsilon.$$
Thus, $G^{(2)} = H_{2,st}^*G^{(1)}H_{1,st}$. Note that $H_{2,st}$ and $H_{1,st}$ are complex unitary matrices. It follows from complex matrix theory that the Hermitian matrices
$\left(G^{(1)}\right)^* G^{(1)}$ and $\left(G^{(2)}\right)^* G^{(2)}$ will have the same eigenvalues, whose square roots are exactly $D_{11}, \cdots, D_{r'r'}$ as displayed in \eqref{sigmat}. The desired uniqueness of the claim is obtained. This completes the proof. \end{proof}

\section{Ranks of Dual Complex Matrices}
In Theorem \ref{SVD.1}, the dual numbers $\mu_1, \cdots, \mu_t$ and possibly $\mu_{t+1} = \cdots = \mu_{\min \{m, n \} }=0$, if $t < \min \{m, n\}$, are called the singular values of $A$, the integer $t$ is called the rank of $A$, and the integer $r$ is called the appreciable rank of $A$.   We denote the rank of $A$ by Rank$(A)$, and the appreciable rank of $A$ by ARank$(A)$.

\begin{Prop} \label{pp4.2}
Suppose that $U \in {\mathbb{DC}}^{m \times m}$ and $V \in {\mathbb{DC}}^{n \times n}$ are unitary, and $A \in {\mathbb {DC}}^{m \times n}$.   Then
\begin{equation} \label{eq10.2}
{\rm Rank}(UAV) = {\rm Rank}(A),
\end{equation}
and
\begin{equation} \label{eq11.2}
{\rm ARank}(UAV) = {\rm ARank}(A).
\end{equation}
\end{Prop}
\begin{proof}
By Theorem \ref{SVD.1}, there are  a dual complex unitary matrix $V_A \in {\mathbb {DC}}^{m \times m}$ and a dual complex unitary matrix $U_A \in {\mathbb {DC}}^{n \times n}$ such that (\ref{e20.1}) holds.  Let $W = UAV$.  Then
$$(V_AU^*)W(V^*U_A) = D,$$
where $D$ is the diagonal matrix in (\ref{e20.1}).  Since $V_AU^*$ and $V^*U_A$ are unitary and the form of $\Sigma$ is unique, we have (\ref{eq10.2}) and (\ref{eq11.2}).
\end{proof}

\begin{Prop} Suppose that $A \in {\mathbb{DC}}^{n \times m}$. Then ${\rm Rank}(A)={\rm Rank}(A^*)={\rm Rank}(A^\top)= {\rm Rank}(\bar{A}) $, and ${\rm ARank}(A)={\rm ARank}(A^*)={\rm ARank}(A^\top)= {\rm ARank}(\bar{A})$.
\end{Prop}
\begin{proof} Assume that the SVD of $A$ is $A = U\Sigma V^*$ with unitary dual complex matrices $U$ and $V$. By virtue of \eqref{commu}, we have $A^\top = (V^*)^\top\Sigma^\top U^\top$ and $A^* = V \Sigma^* U^*$. Note that $U^\top$ and $(V^*)^\top$ are unitary dual complex matrices, and all the diagonal entries of $\Sigma$ are nonnegative dual numbers. Combining with the fact $\bar{A} = (A^*)^\top$, we can easily obtain all the desired assertions.
\end{proof}

The proof of the following proposition is direct.   We omit its proof here.

\begin{Prop} \label{pp4.3}
Suppose that $A \in {\mathbb {DC}}^{m \times n}$ can be written as
$$A = \begin{bmatrix} A_r \\ A_{s-r}\epsilon \\ O \end{bmatrix},$$
where $A_r \in {\mathbb {DC}}^{r \times n}$ and $A_{s-r} \in {\mathbb {C}}^{s-r \times n}$, for $0 \le r \le s \le m$.   Then
$${\rm Rank}(A) \le s, \ \ {\rm ARank}(A) \le r.$$
\end{Prop}

We now prove the following theorem.

\begin{Thm} \label{Rank}
Suppose that $A \in {\mathbb {DC}}^{m \times n}$ and $B \in {\mathbb {DC}}^{n \times p}$.   Then,
\begin{equation} \label{eq12}
{\rm Rank}(AB) \le \min \{{\rm Rank}(A), {\rm Rank}(B)\},
\end{equation}
and
\begin{equation} \label{eq13}
{\rm ARank}(AB) \le \min \{{\rm ARank}(A), {\rm ARank}(B)\}.
\end{equation}
\end{Thm}
\begin{proof}
It suffices to prove
\begin{equation} \label{eq14}
{\rm Rank}(AB) \le {\rm Rank}(A),
\end{equation}
and
\begin{equation} \label{eq15}
{\rm ARank}(AB) \le {\rm ARank}(A).
\end{equation}
The other parts hold similarly.   By Theorem \ref{SVD.1}, there are unitary matrices $U_A$ and $V_A$ such that (\ref{e20.1}) holds.    Then (\ref{eq14}) and (\ref{eq15}) are equivalent to
\begin{equation} \label{eq16}
{\rm Rank}(U_ADV_AB) \le {\rm Rank}(U_ADV_A),
\end{equation}
and
\begin{equation} \label{eq17}
{\rm ARank}(U_ADV_AB) \le {\rm ARank}(U_ADV_A).
\end{equation}
By Proposition \ref{pp4.2}, it suffices to prove
\begin{equation} \label{eq18}
{\rm Rank}(DC) \le {\rm Rank}(D),
\end{equation}
and
\begin{equation} \label{eq19}
{\rm ARank}(DC) \le {\rm ARank}(D),
\end{equation}
where $D$ is the diagonal matrix in (\ref{e20.1}), $C = V_AB$.   Now the conclusion follows from Proposition \ref{pp4.3}.
\end{proof}

\begin{Cor}
Suppose that $A \in {\mathbb {DC}}^{m \times n}$.   Then
\begin{equation} \label{eq20}
{\rm Rank}(A) = \max \{ r : A = BC, B \in {\mathbb {DC}}^{m \times r}, C \in {\mathbb {DC}}^{r \times n} \}.
\end{equation}
\end{Cor}
\begin{proof}
By Theorem \ref{Rank}, the left side of (\ref{eq20}) is less than or equal to the right hand side of (\ref{eq20}).  Combining with the SVD of $A$ as stated in Theorem \ref{SVD.1}, the equality follows.
\end{proof}

\begin{Cor}\label{cor6.6} Suppose that $A \in {\mathbb {DC}}^{m \times n}$ and $B \in {\mathbb {DC}}^{q \times n}$. Let $C = \left[\begin{array}{c}
                                                                                                                            A \\
                                                                                                                            B
                                                                                                                          \end{array}\right]\in {\mathbb {DC}}^{(m+q) \times n}$. Then
\begin{equation}\label{augment}
{\rm Rank}(C) \leq {\rm Rank}(A)+{\rm Rank}(B), ~~{\rm ARank}(C) \leq {\rm ARank}(A)+{\rm ARank}(B).
\end{equation}
\end{Cor}
\begin{proof} Let $A = U_A \Sigma_A V_A^*$ and $B = U_B \Sigma_B V_B^*$ be the singular value decompositions of $A$ and $B$, respectively. By setting $U = {\rm diag}(U_A, U_B)$, $V= {\rm diag}(V_A, V_B)$, and $W = \left[ \begin{array}{cc}
                                 A & O \\
                                 O & B \\
                               \end{array}
                             \right] $,  one can verify that $U$ and $V$ are unitary dual complex matrices, and
$$ W = U\left[ \begin{array}{cc}
                                 \Sigma_A & O \\
                                 O & \Sigma_B\\
                               \end{array}
                             \right] V^*.$$
Thus, ${\rm Rank}(W) = {\rm Rank}(A)+{\rm Rank}(B)$ and ${\rm ARank}(W) = {\rm ARank}(A)+{\rm ARank}(B)$.
Note that $C = W\left[\begin{array}{c}
                                      I_n \\
                                      I_n
                                    \end{array} \right]$.
The assertions follow readily from Theorem \ref{Rank}.

\end{proof}

\begin{Cor}\label{cor6.7} Suppose that $A$, $B \in {\mathbb {DC}}^{m \times n}$. Then
\begin{equation}\label{sum1}
 {\rm Rank}(A)-{\rm Rank}(B) \leq {\rm Rank}(A+B) \leq {\rm Rank}(A)+{\rm Rank}(B),
\end{equation}
\begin{equation}\label{sum2}
 {\rm ARank}(A)-{\rm ARank}(B) \leq{\rm ARank}(A+B) \leq {\rm ARank}(A)+{\rm ARank}(B).
\end{equation}
\end{Cor}
\begin{proof} Note that $A+B = \left[I_m, I_m\right]
\left[\begin{array}{c}    A \\
     B \end{array}\right]
$. The assertion ${\rm Rank}(A+B) \leq {\rm Rank}(A)+{\rm Rank}(B)$ follows directly from Theorem \ref{Rank} and Corollary \ref{cor6.6}. By writing $A = (A+B)+(-B)$, together with the fact Rank$(-B)=$ Rank$(B)$, we have Rank$(A) \leq $ Rank$(A+B)+$ Rank$(B)$. The inequalities for the appreciable rank can be proved in the same manner.
\end{proof}

The following theorem indicates that the standard part of the singular values of a dual complex matrix are exactly the singular values of the standard part of that dual complex matrix, and the appreciable rank of a dual complex matrix is exactly the rank of the standard part of that dual complex matrix.

\begin{Thm}
A dual complex matrix $A = A_{st} + A_\I\epsilon \in {\mathbb {DC}}^{m \times n}$ has singular values
$\sigma_{1, st} + \sigma_{1, \I}\epsilon \ge \cdots \ge \sigma_{\min\{m, n\}, st} + \sigma_{\min \{ m, n\}, \I}\epsilon$ if and only if the complex matrix $A_{st}$ has singular values $\sigma_{1, st} \ge \cdots \ge \sigma_{\min \{ m, n\}, st}$.   We also have
\begin{equation} \label{eq21}
{\rm ARank}(A) = {\rm Rank}(A_{st}).
\end{equation}
\end{Thm}
\begin{proof}
In (\ref{e20.1}), write $U_A = U_{A, st}+ U_{A, \I}\epsilon$, $D = D_{st}+ D_{\I}\epsilon$ and
$V_A = V_{A, st}+ V_{A, \I}\epsilon$.  Then we have
\begin{equation} \label{eq22}
A_{st} = U_{A, st}D_{st}V_{A, st}.
\end{equation}
By Proposition \ref{pp3.1}, $U_{A, st}$ and $V_{A, st}$ are unitary.   Then (\ref{eq22}) is a singular value decomposition of the complex matrix $A_{st}$.   The conclusions follow from this.
\end{proof}

\section{An Eckart-Young Like Theorem for Dual Complex Matrices}\label{Sec.7}
For large scale dual complex matrices, it is of great interest to extract its low rank approximation, for the sake of data reduction. In this section, the low %appreciable
rank approximation of a given dual complex matrix will be discussed, by establishing an Eckart-Young like theorem analogous to the case of complex matrices.

\begin{Thm}\label{EY-Thm}
Suppose that $A\in \mathbb{DC}^{m\times n}$ has singular value decomposition $A=\sum_{j=1}^r\mu_j{\bf u}_j{\bf v}_j^*$. If $k\leq r$. Then the matrix $A_k=\sum_{j=1}^k\mu_j{\bf u}_j{\bf v}_j^*$ satisfies
$$
\|A-A_k\|_F\leq \|A-B\|_F
$$
for any $B\in \mathbb{DC}^{m\times n}$ with rank at most $k$.
\end{Thm}

\begin{proof}
Recall that $\|UDV \|_F=\|D\|_F$ for any dual complex matrix $D = D_{st}+D_{\I} \epsilon\in \mathbb{DC}^{m\times n}$ and dual complex unitary matrix $U\in \mathbb{DC}^{m\times m}$ and dual complex unitary matrix $V\in \mathbb{DC}^{n\times n}$.  If $D$ has nonzero singular values $\sigma_1(D)\geq\sigma_2(D)\geq\ldots\geq\sigma_r(D)\geq 0$, then
$$\|D\|_F=\left\|\left(
                         \begin{array}{c}
                           \sigma_1(D) \\
                           \sigma_2(D) \\
                           \vdots \\
                           \sigma_r(D) \\
                         \end{array}
                       \right)
\right\|_2.% =\left\{
           % \begin{array}{ll}
           %   \left(\sum_{i=1}^r\sigma_i^2(D)\right)^\frac{1}{2}, & \hbox{if $D_{st}\neq O$;} \\
            %   \|D_{\I}\|_F ~\epsilon, & \hbox{otherwise.}
            %\end{array}
          %\right.
$$

Take any $B\in \mathbb{DC}^{m\times n}$ with rank at most $k\leq r$. Let $C=A-B$. Denote $l= {\rm Rank}(C)$, and suppose that $C$ has singular value decomposition $C=\sum_{j=1}^l\gamma_j{\bf x}_j{\bf y}_j^*$ with $\gamma_1\geq\gamma_2\geq\ldots\geq\gamma_l\geq0$ and orthonormal sets $\{{\bf x}_1,{\bf x}_2,\ldots,{\bf x}_l\}\subset \mathbb{DC}^m$ and $\{{\bf y}_1,{\bf y}_2,\ldots,{\bf y}_l\}\subset \mathbb{DC}^n$. By
Corollary \ref{c3.4}, the set $\{{\bf y}_1,{\bf y}_2,\ldots,{\bf y}_l\}\subset \mathbb{DC}^n$ can be extended to an orthonormal basis $\{{\bf y}_1,{\bf y}_2,\ldots,{\bf y}_l,\ldots,{\bf y}_n\}$ for $\mathbb{DC}^n$. Then every unit vector ${\bf z}\in \mathbb{DC}^n$ can be written as ${\bf z}=\sum_{i=1}^na_i{\bf y}_i$
for a unit vector $(a_1,a_2\ldots,a_n)^\top\in \mathbb{DC}^n$ so that
$$
\|C{\bf z}\|_2 =\left\|\left(\sum_{j=1}^l\gamma_j{\bf x}_j{\bf y}_j^*\right)\left(\sum_{i=1}^na_i{\bf y}_i\right)\right\|_2
=\left\|\sum_{j=1}^l\gamma_ja_j{\bf x}_j\right\|_2 =\left\|\left(
                         \begin{array}{c}
                           \gamma_1a_1 \\
                           \vdots \\
                           \gamma_la_l \\
                         \end{array}
                       \right)
\right\|_2,%\left(\sum_{j=1}^l|\gamma_ja_j|^2\right)^{\frac{1}{2}}.
$$
where the last equality follows from Proposition \ref{pp3.2}.
By this, together with dual numbers $\gamma_1\geq \gamma_2\geq \cdots \geq \gamma_l\geq 0$, we have $\gamma_1=|\gamma_1|=\|C{\bf y}_1\|_2$ and for any unit vector ${\bf z}\in \mathbb{DC}^n$,
%\blue{(i) if $C{\bf z}$ is appreciable, then}
$$\|C{\bf {\bf z}}\|_2 %=\left(\sum_{j=1}^l|\gamma_j|^2|a_j|^2\right)^{\frac{1}{2}}
  %\left(\sum_{j=1}^l|\gamma_1|^2|a_j|^2\right)^{\frac{1}{2}}
\leq\gamma_1\left\|\left(
                         \begin{array}{c}
                            a_1 \\
                           \vdots \\
                            a_l \\
                         \end{array}
                       \right)
\right\|_2 \leq \gamma_1.
$$
%\blue{(ii) if $C{\bf z}$ is infinitesimal, then \red{$\|C{\bf {\bf z}}\|_{2}\leq |\gamma_1\|}
Therefore, we have
\begin{equation}\label{Eq-1}
\gamma_1=\max\left\{\|C{\bf z}\|_2~:~{\bf z}\in \mathbb{DC}^n~{\rm with}~\|{\bf z}\|_2=1\right\},
\end{equation}
and
\begin{equation}\label{Eq-2}
\gamma_j=\max\{\|C{\bf z}\|_2~:~{\bf z}\in \mathbb{DC}^n~{\rm with}~\|{\bf z}\|_2=1~{\rm and}~{\bf y}_1^*{\bf z}=\ldots={\bf y}_{j-1}^*{\bf z}=0\}
\end{equation}
for $j=2,\ldots, l$.
Since $B$ has rank at most $k$, the matrix $B[{\bf v}_1,\ldots,{\bf v}_{k+1}]\in \mathbb{DC}^{m\times (k+1)}$ has also rank at most $k$ from Theorem \ref{Rank}. Hence, there is a unit vector $\tilde{{\bf w}}=(w_1,w_2\ldots,w_{k+1})^\top \in \mathbb{DC}^{k+1}$ satisfying $B[{\bf v}_1,\ldots,{\bf v}_{k+1}]\tilde{{\bf w}}=0$. %(it follows from SVD of $B[{\bf v}_1,\ldots,{\bf v}_{k+1}]$).
Consider the unit vector
$\tilde{\bf z}=[{\bf v}_1,\ldots,{\bf v}_{k+1}]\tilde{{\bf w}}\in \mathbb{DC}^{n}$. By (\ref{Eq-1}), we have
\begin{equation}\label{Eq-3}
\gamma_1\geq\|C\tilde{\bf z}\|_2=\left\|\sum_{i=1}^{k+1}\mu_iw_i{\bf u}_i\right\|_2=\left\|\left(
                         \begin{array}{c}
                           \mu_1w_1 \\
                           \vdots \\
                           \mu_{k+1}w_{k+1} \\
                         \end{array}
                       \right)
\right\|_2 \geq \mu_{k+1}\left\| \tilde{{\bf w}}\right\|_2 = \mu_{k+1}.
%\left(\sum_{i=1}^{k+1}|\mu_i|^2|w_i|^2\right)^{\frac{1}{2}}\geq |\mu_{k+1}|\left(\sum_{i=1}^{k+1}|w_i|^2\right)^{\frac{1}{2}}=|\mu_{k+1}|=\mu_{k+1}.
\end{equation}

It is known from Corollary \ref{cor6.7} that $l\geq r-k$. If $l> r-k$, then set $\mu_{r+1}=\cdots = \mu_{l+k} =0$. Now we prove $\gamma_j\geq\mu_{k+j}$ for every $j=2,\ldots, l$. Let
$$
B_j=\left[\begin{array}{c}
B\\
{\bf y}_1^*\\
\vdots\\
{\bf y}_{j-1}^*
\end{array}\right].
$$
By applying Corollary \ref{cor6.6}, we know that $B_{j}\in \mathbb{DC}^{(m+j-1)\times n}$ has rank at most $k+j-1$, and so does the dual complex matrix $B_{j}[{\bf v}_1,\ldots,{\bf v}_{k+j}]\in \mathbb{DC}^{(m+j-1)\times (k+j)}$, due to Theorem \ref{Rank}. It is similar to the case $j=1$, that there is a unit vector $\tilde{{\bf w}}_j=(w_1,w_2\ldots,w_{k+j})^\top \in \mathbb{DC}^{k+j}$ satisfying $B_j[{\bf v}_1,\ldots,{\bf v}_{k+j}]\tilde{{\bf w}}_j=0$. Consider the unit vector
$\tilde{\bf z}=[{\bf v}_1,\ldots,{\bf v}_{k+j}]\tilde{{\bf w}}_j\in \mathbb{DC}^{n}$. It is obvious that ${\bf y}_1^*\tilde{{\bf z}}=\ldots={\bf y}_{j-1}^*\tilde{{\bf z}}=0$. Consequently, by (\ref{Eq-2}), we have
\begin{equation}\label{Eq-4}
\gamma_j\geq\|C\tilde{\bf z}\|_2=\left\|\sum_{i=1}^{k+j}\mu_iw_i{\bf u}_i\right\|_2=\left\|\left(
                         \begin{array}{c}
                           \mu_1w_1 \\
                           \vdots \\
                           \mu_{k+j}w_{k+j} \\
                         \end{array}
                       \right)
\right\|_2 \geq \mu_{k+j}\left\| \tilde{{\bf w}}_j\right\|_2 = \mu_{k+j}\geq 0,
%\left(\sum_{i=1}^{k+j}|\mu_i|^2|w_i|^2\right)^{\frac{1}{2}}\geq |\mu_{k+j}|\left(\sum_{i=1}^{k+j}|w_i|^2\right)^{\frac{1}{2}}=|\mu_{k+j}|=\mu_{k+j},
\end{equation} for all $j=1,\cdots, l$. By (\ref{Eq-3}) and (\ref{Eq-4}), we have
$$\|A-A_k\|_F=\left\|\sum_{i=k+1}^r\mu_i{\bf u}_i{\bf v}_i^*\right\|_F= \left\|\left(
                         \begin{array}{c}
                           \mu_{k+1} \\
                           \vdots \\
                           \mu_{r}  \\
                         \end{array}
                       \right)\right\|_2 \leq \left\|\left(
                         \begin{array}{c}
                           \gamma_{1} \\
                           \vdots \\
                           \gamma_{r-k}  \\
                         \end{array}
                       \right)\right\|_2 \leq  \left\|\left(
                         \begin{array}{c}
                           \gamma_{1} \\
                           \vdots \\
                           \gamma_{r-k}\\
                           \vdots \\
                           \gamma_{l}  \\
                         \end{array}
                       \right)\right\|_2 = \|A-B\|_F,
$$
%\left(\sum_{i=k+1}^r\mu_i^2\right)^\frac{1}{2}\leq\left(\sum_{i=1}^{r-k}\gamma_i^2\right)^\frac{1}{2}\leq\left(\sum_{i=1}^{l}\gamma_i^2\right)^\frac{1}{2}=\|A-B\|_F,$$
which  means that the desired result holds.
\end{proof}

\section{Numerical Experiments}

\subsection{Truncated SVD for Low-Rank Approximation}
This subsection is devoted to the approach of finding the best low-rank approximation of a dual complex matrix $A\in {\mathbb{DC}}^{m\times n}$ based on the truncated SVD. Before proceeding, we first consider the unitary decomposition of a Hermitian dual complex matrix of the form $B = A^* A$ with $A\in {\mathbb {DC}}^{m\times n}$. The algorithmic framework is stated in Algorithm \ref{Alg1}.

\begin{algorithm}%[!th]
\SetAlgoLined
{\bf Input:} $A = A_{st}+ A_{\I} \epsilon \in  {\mathbb {DC}}^{m\times n}$ \\
{\bf Output:} Unitary $U_B = (U_B)_{st}+ (U_{B})_{\I} \epsilon \in  {\mathbb {DC}}^{n\times n}$, diagonal $\Sigma_B = (\Sigma_B)_{st}+ (\Sigma_{B})_{\I} \epsilon \in  {\mathbb {D}}^{n\times n}$
\begin{itemize}
  \item Compute the eigenvalue decomposition of $B_{st}: = A_{st}^* A_{st}$: %{\tt $[S,D] = \mathrm{eig}(B_{st})$}, i.e.,
            $$B_{st} = SDS^*,~~S\in {\mathbb C}^{n\times n}\mbox{~unitary}, ~~D = \mathrm{diag}(\lambda_1 I_{k_1},\cdots, \lambda_r I_{k_r}) \in {\mathbb R}^{n\times n} \mathrm{~with~} \lambda_1 > \cdots > \lambda_r\geq 0$$
  \item Compute $M = S^* B_{\I} S = \left(
                             \begin{array}{cccc}
                               C_{11} & C_{12}& \cdots & C_{1r} \\
                               C_{12}^*& C_{22} & \cdots & C_{2r} \\
                               \vdots & \vdots & \vdots & \vdots\\
                               C_{1r}^* & C_{2r}^* & \cdots & C_{rr} \\
                             \end{array}
                           \right)$
  \item  Compute $P_{\I} = \left(
                             \begin{array}{cccc}
                               O & \frac{C_{12}}{\lambda_1-\lambda_2}  & \cdots & \frac{C_{1r}}{\lambda_1-\lambda_r}  \\
                               -\frac{C_{12}^*}{\lambda_1-\lambda_2} & O & \cdots &   \frac{C_{2r}}{\lambda_2-\lambda_r}  \\
                               \vdots & \vdots & \vdots & \vdots \\
                                -\frac{C_{12}^*}{\lambda_1-\lambda_r}  &  -\frac{C_{2r}^*}{\lambda_2-\lambda_r}   & \cdots & O\\
                             \end{array}
                           \right)$
  \item For $i = 1,\cdots, r$, perform the eigenvalue decomposition of $C_{ii}$:
  $$ C_{ii} = V_i D_i V_i^*,~~ V_i\in {\mathbb C}^{k_i\times k_i}~\mathrm{unitary},~D_i = \mathrm{diag} (\lambda_{i,1},\cdots, \lambda_{i,k_i})$$
  \item Set $\hat{V}_B = \mathrm{diag}(V_1,\cdots, V_r)$, and compute
  $$(U_B)_{st} = S \hat{V}_B,~~(U_B)_{\I} = S P_{\I}^* \hat{V}_B.$$
  \item Set $$(\Sigma_B)_{st} = \mathrm{diag}\left(\underbrace{\lambda_1,\ldots, \lambda_1}_{k_1~\mathrm{times}}, \ldots, \underbrace{\lambda_r,\ldots, \lambda_r}_{k_r~\mathrm{times}}\right)$$ and
  $$(\Sigma_B)_{\I} = \mathrm{diag}\left(\lambda_{1,1},\ldots, \lambda_{1, k_1}, \ldots,  \lambda_{r,1},\ldots, \lambda_{r, k_r}\right)$$
  \end{itemize}
\caption{ Unitary Decomposition of $B = A^*A\in {\mathbb {DC}}^{n\times n}$ \label{Alg1}}
\end{algorithm}

Based on the unitary decomposition as presented in Algorithm \ref{Alg1}, we are in the position to present the procedure of the singular value decomposition of any given dual complex matrix $A\in {\mathbb{DC}}^{m\times n}$, see Algorithm \ref{Alg2}.

According to the Eckart-Young like theorem as stated in Theorem \ref{EY-Thm}, the low-rank approximation of $A\in {\mathbb{DC}}^{m\times n}$ can be obtained by the truncated SVD. Specifically, given $A\in {\mathbb{DC}}^{m\times n}$ and a positive integer $k<\min\{m,n\}$, Algorithm \ref{Alg2} can produce the SVD of $A$ in terms of
$A = V \Sigma U^*$. Then the best low-rank approximation of $A$ with rank no more than $k$, termed as $A_k$, can be obtained by
\begin{equation}\label{trunc}
A_k: = V_{:,1:k}\Sigma_{1:k,1:k} \left(U_{:,1:k}\right)^*.
\end{equation}
Note that such a low-rank approximation may not be unique. For example, when $\Sigma_{kk} = \Sigma_{(k+1)(k+1)}$, then both $A_k$ as defined in \eqref{trunc} and \begin{equation}\label{trunc1} \tilde{A}_k := V_{:,\Omega}\Sigma_{\Omega,\Omega} \left(U_{:,\Omega}\right)^* \end{equation} with $\Omega: =\{1,\ldots, k-1,k+1\}$ are best low-rank approximations of $A$ with rank no more than $k$ in the sense that \begin{equation}\label{trunc2}\|A-A_k\|_F = \|A-\tilde{A}_k\|_F.\end{equation}

\begin{algorithm}%[!th]
\SetAlgoLined
%\KwResult{SVD of $A$ by $V^* A U = \Sigma$ }
 {\bf Input:} $A = A_{st}+ A_{\I} \epsilon \in  {\mathbb {DC}}^{m\times n}$ %$\bfz^0\in\F,\beta\in(0,1)$ and $\rho,\varepsilon, {\tt tol}_0>0$.  Set $\ell :=0$.
 \\
 {\bf Output:} \\
 %The appreciable rank of $A$: aRank;  \\
 The unitary matrices: $U = U_{st} + U_{\I}\epsilon \in {\mathbb {DC}}^{n\times n}$, $V = V_{st}+ V_{\I}\epsilon\in {\mathbb {DC}}^{m\times m}$;\\ The rectangular diagonal matrix of all the singular values: $\Sigma = \Sigma_{st}+ \Sigma_{\I}\epsilon \in {\mathbb {D}}^{m\times n}$. \vspace{0.3cm}\\
%\parbox{.45\textwidth}{  }\\
 %\While{${\tt tol}_\ell > \varepsilon$}{
  \underline{Step 1:}  Decompose $B := A^* A\in {\mathbb {DC}}^{n\times n}$ by {\bf Algorithm \ref{Alg1}} to get $U_B = (U_B)_{st} + (U_B)_{\I} \epsilon \in {\mathbb {DC}}^{n\times n}$ and $\Sigma_B = (\Sigma_B)_{st}+ (\Sigma_{B})_{\I} \epsilon \in  {\mathbb {D}}^{n\times n}$ such that  $B = U_B \Sigma_B U_B^*$ \\ %\\
 %\parbox{.45\textwidth}{  }\\
   % $ \bfu^{\ell}:=\bfz^\ell(\tau_\ell)$ and $\bfz^{\ell+1}=\bfu^{\ell}$.
  \underline{Step 2:} Set $\lambda = \left(\lambda_1,\ldots, \lambda_r\right)$ including all distinct diagonal entries of $(\Sigma_B)_{st}$ and compute \begin{equation*}
      s = \|\lambda\|_0 =\sharp \{i: \lambda_i \neq 0, i\in \{1,\ldots, r\}\}, \mathrm{~and~} r_a=\sum\limits_{i=1}^s k_i.
  \end{equation*}
  Set \begin{eqnarray*}
      U^1_{st} = (U_B)_{st}(:, 1:r_a),& ~~U^2_{st} = (U_B)_{st}(:, r_a+1:n) \\
      U^1_{\I} = (U_B)_{\I}(:, 1:r_a),& ~~U^2_{\I} = (U_B)_{\I}(:, r_a+1:n)
  \end{eqnarray*}
  Compute
  $$
  \Sigma^1_{st} = \mathrm{diag}\left(\underbrace{\sqrt{\lambda_1},\ldots, \sqrt{\lambda_1}}_{k_1~\mathrm{times}}, \ldots, \underbrace{\sqrt{\lambda_s},\ldots, \sqrt{\lambda_s}}_{k_s~\mathrm{times}}\right)$$
  $$\Sigma^1_{\I} = \mathrm{diag}\left(\frac{\lambda_{1,1}}{2\sqrt{\lambda_1}},\ldots, \frac{\lambda_{1,k_1}}{2\sqrt{\lambda_1}}, \ldots, \frac{\lambda_{s,1}}{2\sqrt{\lambda_s}},\ldots, \frac{\lambda_{s,k_s}}{2\sqrt{\lambda_s}}\right)$$
    $$\Sigma_{\I}^{-} = \mathrm{diag}\left(\frac{\lambda_{1,1}}{-2\lambda_1\sqrt{\lambda_1}},\ldots, \frac{\lambda_{1,k_1}}{-2\lambda_1\sqrt{\lambda_1}}, \ldots, \frac{\lambda_{s,1}}{-2\lambda_s\sqrt{\lambda_s}},\ldots, \frac{\lambda_{s,k_s}}{-2\lambda_s\sqrt{\lambda_s}}\right)$$
    Compute $$V^1_{st} = A_{st}U^1_{st}(\Sigma^1_{st})^{-1},~~V^1_{\I} = A_{st}U^1_{st}\Sigma_{\I}^{-}+ A_{st}U^1_{\I}(\Sigma^1_{st})^{-1}+ A_{\I}U^1_{st}(\Sigma^1_{st})^{-1}$$

  \underline{Step 3:} For $V^1 = V^1_{st} + V^1_{\I}\epsilon \in {\mathbb {DC}}^{m\times r_a}$, find $V^2 = V^2_{st} + V^2_{\I}\epsilon \in {\mathbb {DC}}^{m\times (m-r_a)}$ such that $(V^1, V^2)$ is unitary according to Corollary \ref{c3.4}.

  \underline{Step 4:} Set
  $$ G = (V^2_{st})^* A_{st} U^2_{\I} + (V^2_{st})^* A_{\I} U^2_{st}  $$
  Decompose $G \in {\mathbb {C}}^{(m-r_a)\times (n-r_a)}$ by classical singular value decomposition to get unitary matrices $W_1 \in {\mathbb {C}}^{(m-r_a)\times (m-r_a)}$, $W_2 \in {\mathbb {C}}^{(n-r_a)\times (n-r_a)}$ and rectangular diagonal matrix $\Sigma_G \in {\mathbb{R}}^{(m-r_a)\times (n-r_a)}$ with diagonals in nonincreasing order such that  $ W_1^* G W_2= \Sigma_G$. Compute
  $$ U_{st} = (U^1_{st}, U^2_{st} W_2),\quad  U_{\I} = (U^1_{\I}, U^2_{\I} W_2),\quad  V_{st} = (V^1_{st}, V^2_{st} W_1), \quad  V_{\I} = (V^1_{\I}, V^2_{\I} W_1), $$
  $$ \Sigma_{st} = \mathrm{diag}\left( \Sigma^1_{st}, O_{(m-r_a)\times (n-r_a)} \right),  \quad \Sigma_{\I} = \mathrm{diag}\left( \Sigma^1_{\I}, \Sigma_G \right)
  $$

 % $$W = A_{st} U_{st},~V_{st} = W Sig_INVst;
%VAI  = W*Sig_INVI+ (Ast*UAI+ AI*UAst)*Sig_INVst; $$

  %\If{$T^\ell= \widetilde T^\ell$ and $\Gamma_+^\ell= \widetilde \Gamma_+^\ell$}{
  % \parbox{.42\textwidth}{\begin{eqnarray}\label{Newton-descent-property}
	%\bfv^\ell =  {\rm argmin}~\{f\left(\bfz\right):\bfz\in \Omega(\bfz^\ell)\}.
%\end{eqnarray}}
%\\
  % If %$f(\bfv^\ell)\leq f(\bfu^\ell) - {\rho}\|\bfv^\ell-\bfu^\ell\|^2$, then set $\bfz^{\ell+1}=\bfv^{\ell}$.
%   %}
%Compute %${\tt tol}_\ell:=\|\bfu^{\ell}-\bfz^{\ell}\|$ and set $\ell:= \ell+1$.
%% }
% Output solution %$\overline{\bfx}=\bfx^\ell/\|\bfx^\ell\|.$
\caption{ Singular Value Decomposition of $A\in {\mathbb {DC}}^{m\times n}$ in terms of  $A = V\Sigma U^*$ \label{Alg2} }
\end{algorithm}

%\newpage

\subsection{Numerical Results}
In this subsection, we present some numerical results to show the efficiency of proposed algorithms. All the codes are written in Python 3.9.5. The numerical experiments were done on a Macbook notebook with an Intel m3 dual-core processor running at 1.2GHz and 8GB of RAM.

%\noindent
%\textbf{Example 8.1.}
%First we use a toy example to show the efficiency of proposed algorithms.
\begin{Exa}\label{Ex8.1} The dual complex matrix $A=A_{st} + A_\I\epsilon \in {\mathbb {DC}}^{8 \times 4}$ is given by
$$ A_{st} = \left[
\begin{array}{rrrr}
 0.0942+0.5476\ii & -0.0558-0.8419\ii & -0.0110+0.3480\ii & -0.5487-0.5106\ii \\
 0.8862+0.4817\ii &  0.9710-0.9956\ii &  0.8044-0.2831\ii &  0.4500+0.2713\ii \\
-0.7160+0.1136\ii &  0.1822-0.1160\ii &  0.1444-0.7116\ii &  0.4144+0.9641\ii \\
 0.4003-0.9425\ii & -0.5649+0.1447\ii & -0.8047+0.5493\ii &  0.7744-0.1893\ii \\
-0.8300-0.4158\ii &  0.0739-0.8446\ii &  0.8513-0.3221\ii & -0.2451-0.4130\ii \\
-0.1672+0.7633\ii & -0.1062+0.0883\ii & -0.9567+0.5460\ii &  0.1183-0.2970\ii \\
 0.9717-0.6890\ii &  0.9410-0.8403\ii &  0.5039-0.7443\ii &  0.4999+0.8501\ii \\
-0.0867-0.5021\ii &  0.2065-0.2641\ii & -0.8358-0.5494\ii &  0.7137+0.8727\ii
\end{array} \right]
$$
and
$$ A_{\I} = \left[
\begin{array}{rrrr}
-0.5010-0.8953\ii & 0.6144-0.1153\ii & -0.2793+0.2216\ii & -0.9906-0.8764\ii \\
 0.9656-0.5856\ii & 0.0786-0.3271\ii & -0.6626-0.9692\ii & -0.6339+0.2159\ii \\
 0.3707-0.9521\ii &-0.6612+0.7596\ii &  0.3158+0.8447\ii &  0.2254-0.8259\ii \\
 0.8779+0.8864\ii &-0.1376+0.7962\ii &  0.0601-0.8183\ii &  0.8345+0.7047\ii \\
-0.2593-0.7807\ii & 0.5011-0.3864\ii &  0.7560-0.1174\ii &  0.1853+0.1771\ii \\
-0.5777+0.2567\ii & 0.8656+0.8806\ii & -0.8558+0.0361\ii & -0.0326-0.7419\ii \\
 0.5477+0.7905\ii & 0.9992-0.0610\ii & -0.1922+0.5463\ii &  0.7753-0.0609\ii \\
 0.4876+0.6894\ii &-0.5506-0.5568\ii & -0.0593-0.6293\ii & -0.6266+0.4538\ii
\end{array} \right].
$$
%$$ A_{st} = \left[
%\begin{array}{rrrr}
%-0.4247 - 0.3146\ii & -0.1501 - 0.3392\ii &  0.5004 + 0.1508\ii & 0.1967 + 0.1418\ii \\
%-0.1328 + 0.4259\ii &  0.2830 - 0.1029\ii &  0.2243 + 0.0788\ii & 0.0736 + 0.1600\ii \\
%-0.0202 + 0.1545\ii &  0.4248 - 0.2450\ii &  0.2155 - 0.0261\ii & 0.1055 + 0.0494\ii \\
% 0.2160 + 0.0799\ii &  0.0225 - 0.2973\ii &  0.0516 + 0.2842\ii & 0.4268 - 0.2306\ii \\
% 0.2643 + 0.1982\ii & -0.0417 + 0.0381\ii &  0.4333 - 0.1023\ii &-0.1480 + 0.5233\ii \\
%-0.2875 + 0.1396\ii & -0.0398 - 0.0015\ii & -0.0568 + 0.3083\ii & 0.3452 - 0.2669\ii \\
% 0.2280 - 0.2422\ii & -0.3361 + 0.2869\ii &  0.3299 + 0.2863\ii & 0.1288 + 0.1488\ii \\
% 0.1103 + 0.3319\ii & -0.2450 + 0.4269\ii & -0.0350 + 0.2262\ii & 0.3585 + 0.0891\ii
%\end{array} \right]
%$$
%and
%$$ A_{I} = \left[
%\begin{array}{rrrr}
% 0.0441 + 0.1990\ii & -0.3101 + 0.4095\ii & -0.6919 - 0.3509\ii &  0.2873 + 0.2377\ii \\
%-0.9359 - 0.7543\ii & -0.1424 - 0.0103\ii &  0.3101 - 0.0077\ii &  0.3099 + 0.1019\ii \\
%-0.5679 + 0.1230\ii & -0.4117 + 0.1277\ii &  0.1703 + 0.0529\ii &  0.3256 + 0.0208\ii \\
% 0.0337 - 1.2465\ii &  0.3575 - 0.9436\ii &  0.2181 - 0.1483\ii & -0.0796 + 0.2852\ii \\
% 0.4192 - 0.2467\ii & -0.0449 - 0.0731\ii &  0.2384 - 0.7089\ii &  0.3300 - 0.1167\ii \\
% 0.2350 + 0.3058\ii & -0.3591 + 0.2424\ii & -0.6008 + 0.4495\ii &  0.1064 + 0.4633\ii \\
%-0.4722 - 0.7176\ii & -0.2118 - 0.3770\ii & -0.0414 - 0.3052\ii &  0.2302 + 0.2776\ii \\
% 0.6539 + 0.4971\ii & -0.2148 + 0.0481\ii & -0.3575 + 0.2907\ii &  0.1083 + 0.2476\ii
%\end{array} \right].
%$$
By using Algorithm \ref{Alg2}, all the nonzero singular values of dual complex matrix $A$ are positive appreciable dual numbers given by
$$
\Sigma_A = {\rm diag}\left(3.4147+0.5451\epsilon, 2.4280+0.6444\epsilon, 2.1287-0.5667\epsilon, 0.8744+0.4006\epsilon \right).
$$
The corresponding partially unitary matrix $V=V_{st} + V_\I \epsilon \in {\mathbb {DC}}^{8 \times 4}$ is give by
$$ V_{st} = \left[
\begin{array}{rrrr}
 0.0279+0.0919\ii & -0.0116-0.1099\ii & -0.1357-0.5137\ii & -0.4708+0.3136\ii \\
 0.3506-0.3066\ii &  0.2628-0.0759\ii &  0.1530-0.4269\ii &  0.2549+0.0570\ii \\
 0.0351-0.2229\ii & -0.2700+0.2571\ii &  0.3419+0.1245\ii &  0.2949+0.3974\ii \\
-0.3474-0.0418\ii &  0.3878-0.1294\ii & -0.1672+0.3037\ii &  0.0640+0.1271\ii \\
 0.0142-0.0910\ii & -0.4078-0.4301\ii &  0.2274-0.2196\ii & -0.0112+0.1497\ii \\
-0.1792+0.2150\ii &  0.1925+0.2859\ii &  0.0422-0.2238\ii & -0.3141+0.1643\ii \\
 0.2704-0.5472\ii &  0.2167-0.1328\ii &  0.0826+0.0717\ii & -0.1222-0.0123\ii \\
-0.1599-0.3398\ii &  0.0926+0.2520\ii &  0.1853+0.2477\ii & -0.3413+0.2587\ii
\end{array} \right],
$$
$$ V_{\I} = \left[
\begin{array}{rrrr}
 0.0154+0.0519\ii & -0.8883-0.9381\ii & -0.2791-0.0629\ii &  0.2414-0.7158\ii \\
 0.0906-0.2965\ii & -0.5898-0.3918\ii & -0.8754-0.0513\ii &  0.1141-0.2964\ii \\
-0.2370+0.3518\ii &  0.7696-0.6322\ii &  0.5323+0.5947\ii & -0.3348+0.1037\ii \\
 0.2615-0.1980\ii &  0.4132+1.1643\ii & -0.9351-0.2609\ii &  0.5067-0.9837\ii \\
 0.0265-0.2820\ii &  0.0204-0.7098\ii &  0.2755+1.3629\ii &  0.5371+0.3169\ii \\
 0.0266+0.2995\ii & -0.0576-0.3514\ii &  0.4129-0.9124\ii & -0.1488-0.4585\ii \\
 0.1512+0.0934\ii &  0.5073+0.4991\ii & -0.0747-0.2496\ii &  0.0608-0.2841\ii \\
 0.1614+0.0240\ii &  0.3599+0.4377\ii & -0.1807-0.8833\ii & -1.0300+0.4473\ii
\end{array} \right]
$$
and the unitary matrix $U=U_{st} + U_\I \epsilon \in {\mathbb {DC}}^{4 \times 4}$ is given by
$$ U_{st} = \left[
\begin{array}{rrrr}
 0.3251-0.2489\ii &  0.6201           & -0.6262-0.0516\ii &  0.0401-0.2268\ii \\
 0.4908-0.0864\ii &  0.2937-0.0301\ii &  0.6427           & -0.4347-0.2497\ii \\
 0.5901           & -0.2697-0.4888\ii &  0.0600-0.0323\ii &  0.5769+0.0519\ii \\
-0.3634-0.3220\ii &  0.4604+0.0681\ii &  0.4329+0.0089\ii &  0.6001
\end{array} \right],
$$
$$ U_{\I} = \left[
\begin{array}{rrrr}
 0.1761-0.0836\ii & -0.9991+0.8785\ii & -0.5686-0.7263\ii &  0.3247-0.5953\ii \\
 0.2281+0.1830\ii &  0.7788-0.8423\ii & -0.2960-0.7558\ii & -0.0148+0.6665\ii \\
-0.1969-0.2544\ii & -0.3037-0.0170\ii & -0.2460+1.7207\ii &  0.1777-0.1192\ii \\
 0.1222+0.0423\ii &  0.6408-0.2896\ii & -0.3052-0.0871\ii & -0.1405+0.5359\ii
\end{array} \right].
$$
\end{Exa}
%\noindent
%\textbf{Example 8.2.}

\begin{Exa}\label{Ex8.2}
We test dual complex matrices with multiple standard singular values.  The matrix $A=A_{st} + A_\I\epsilon \in {\mathbb {DC}}^{6 \times 4}$ is given by
$$ A_{st} = \left[
\begin{array}{rrrr}
 0.1041+0.3547\ii &  0.0592-0.1799\ii &  0.4438-0.4804\ii &  0.0835+0.4100\ii \\
-0.4672+0.3171\ii & -0.1917-0.3081\ii & -0.0317-0.0438\ii & -0.2192-0.4308\ii \\
-0.4863-0.0343\ii & -0.1600+0.5508\ii & -0.2348-0.0403\ii &  0.6214-0.1158\ii \\
-0.3769+0.4002\ii &  0.2875-0.4984\ii &  0.6286+0.2987\ii & -0.7597-0.1947\ii \\
 0.2583-0.4642\ii & -0.0509-0.3241\ii & -0.2919-0.0071\ii & -0.0091-0.1178\ii \\
 0.4598+0.4482\ii & -0.6406-0.0624\ii & -0.1258-0.2293\ii & -0.5014+0.3828\ii
\end{array} \right]
$$
and
$$ A_{\I} = \left[
\begin{array}{rrrr}
-0.8515+0.6585\ii &  0.1589-0.2629\ii &  0.9587-0.9404\ii &  0.1717-0.5454\ii \\
 0.8687+0.2795\ii &  0.0565+0.5294\ii &  0.2467+0.6030\ii &  0.0594+0.6913\ii \\
-0.7183+0.8306\ii &  0.1741-0.0283\ii & -0.8307-0.8884\ii &  0.7921+0.2311\ii \\
-0.4544-0.3646\ii & -0.3832+0.7387\ii & -0.0662+0.0687\ii &  0.6902-0.0236\ii \\
 0.6920-0.5542\ii & -0.7215-0.3887\ii & -0.0265+0.0162\ii & -0.3195-0.5390\ii \\
 0.8588+0.6675\ii &  0.5337-0.0519\ii &  0.7599-0.1386\ii &  0.1496-0.8450\ii
\end{array} \right].
$$
%$$ A_{st} = \left[
%\begin{array}{rrrr}
% 0.5532-0.0794\ii  &  0.5797-0.3667\ii & 0.3944+0.3037\ii & 0.4737+0.2654\ii \\
% 0.3156-0.5634\ii  &  0.0693+0.3177\ii &-0.601 +0.2677\ii &-0.6527-0.7635\ii \\
% 0.2600+0.4462\ii  &  0.6695-0.2548\ii &-0.6667+0.4427\ii & 0.1003-0.5236\ii \\
%-0.1857+0.4311\ii  &  0.0171+0.5510\ii & 0.3710-0.7093\ii &-0.5104+0.4501\ii \\
% 0.3418+0.2304\ii  & -0.5865+0.2597\ii &-0.0294+0.0599\ii &-0.7272-0.6267\ii \\
% 0.1919-0.0067\ii  &  0.7406-0.7113\ii &-0.4523+0.0259\ii &-0.7116+0.2590\ii
%\end{array} \right]
%$$
%$$ A_{I} = \left[
%\begin{array}{rrrr}
%-0.4656+0.4152\ii & 0.9328-0.6012\ii &-0.9976-0.2982\ii &-0.1428-0.6976\ii \\
%-0.9307+0.6370\ii & 0.2478-0.6137\ii & 0.7165-0.4705\ii & 0.3721-0.7945\ii \\
% 0.5823+0.8629\ii &-0.0697+0.5872\ii &-0.6659-0.8159\ii &-0.2473+0.4852\ii \\
% 0.5711+0.9434\ii &-0.9223-0.1342\ii &-0.5033-0.5046\ii & 0.8041+0.1552\ii \\
%-0.9454-0.5565\ii &-0.2608-0.7558\ii & 0.6601-0.4940\ii &-0.7526+0.0783\ii \\
%-0.4554-0.4826\ii &-0.4252-0.4054\ii & 0.8224-0.5150\ii &-0.9471-0.2174\ii
%\end{array} \right]
%$$
Here the complex matrix $A_{st}$ is constructed with singular values $2, 1, 1, 0$. By using Algorithm \ref{Alg2}, all the nonzero singular values of dual complex matrix $A$ are positive dual numbers given by
$$
\Sigma_A = {\rm diag}\left( 2-0.4551\epsilon, 1-0.4524\epsilon, 1+1.9418\epsilon, 0+ 0.9203\epsilon \right).
$$
The corresponding partially unitary matrix $V=V_{st} + V_\I \epsilon \in {\mathbb {DC}}^{6 \times 4}$ is give by
$$ V_{st} = \left[
\begin{array}{rrrr}
-0.2106-0.2339\ii & -0.3265+0.2413\ii & -0.3539+0.2987\ii & -0.4623+0.5243\ii \\
 0.2230-0.1973\ii &  0.5322+0.0760\ii & -0.0513+0.2022\ii & -0.2601+0.1058\ii \\
 0.0612+0.3576\ii &  0.2741-0.2340\ii & -0.4780+0.3717\ii & -0.0490-0.0799\ii \\
 0.3606-0.5315\ii & -0.1493-0.0002\ii & -0.0812-0.2319\ii &  0.3017+0.2590\ii \\
-0.0463+0.1119\ii &  0.1608+0.3877\ii &  0.4053-0.3000\ii & -0.4113+0.0855\ii \\
-0.3980-0.3085\ii &  0.3325-0.3275\ii &  0.2463+0.0019\ii & -0.2901-0.0679\ii
\end{array} \right],
$$
$$ V_{\I} = \left[
\begin{array}{rrrr}
 0.1371-0.6584\ii &  0.2540+0.4767\ii & -0.1966+0.0528\ii &  0.2648-0.0997\ii \\
-0.2880+0.0937\ii & -0.3750-0.4115\ii &  0.6110-0.5756\ii & -0.1252-0.2210\ii \\
 0.0223+0.1212\ii &  0.0878+0.1483\ii &  0.1702+0.9506\ii & -0.3231-0.0264\ii \\
 0.1348+0.3201\ii &  0.8156+0.2973\ii &  0.2300+0.4261\ii &  0.2807+0.1123\ii \\
-0.2812+0.2160\ii &  0.4647+0.4034\ii & -0.1816+0.1141\ii &  0.0689-0.1729\ii \\
 0.1392-0.1706\ii & -0.1497-0.4576\ii &  0.0619+0.2437\ii & -0.2355-0.1253\ii
\end{array} \right]
$$
%$$ V_{st} = \left[
%\begin{array}{rrrr}
%-0.2896-0.2354\ii & -0.2463-0.0389\ii & -0.5334+0.0072\ii & -0.2387-0.4219\ii \\
% 0.5459+0.1086\ii & -0.2651+0.1990\ii & -0.0654+0.4585\ii &  0.2668-0.0129\ii \\
% 0.1629-0.0204\ii & -0.5699-0.1585\ii &  0.214 -0.3533\ii & -0.0062-0.1629\ii \\
% 0.1837-0.0923\ii &  0.495 -0.2374\ii & -0.1293-0.2285\ii & -0.362 +0.2312\ii \\
% 0.4079+0.2588\ii &  0.041 +0.2592\ii & -0.2697-0.4078\ii &  0.0498+0.2387\ii \\
% 0.1766-0.4659\ii & -0.3281+0.0139\ii &  0.1453+0.0352\ii & -0.4018+0.5116\ii
%\end{array} \right],
%$$
%$$ V_{I} = \left[
%\begin{array}{rrrr}
% 0.1516+0.0788\ii & -0.7597-0.5145\ii &  0.9545-0.1683\ii & -0.6981+0.4683\ii \\
%-0.2443+0.4369\ii &  0.0452-0.1510\ii &  0.2987-0.2500\ii &  0.1864-0.0999\ii \\
% 0.3683+0.0104\ii &  0.304 -0.5483\ii & -0.0573-1.2691\ii &  1.1496+1.7469\ii \\
%-0.132 +0.2501\ii & -0.1519-0.1158\ii & -0.5379-0.2504\ii & -0.1478+0.3837\ii \\
% 0.0113-0.2128\ii &  0.5327+0.3493\ii & -0.1474+0.0728\ii &  0.0204+0.3749\ii \\
% 0.3482-0.2681\ii &  0.4724+0.8294\ii &  0.4059+0.3571\ii & -0.1266-0.0235\ii
%\end{array} \right],
%$$
and the unitary matrix $U=U_{st} + U_\I \epsilon \in {\mathbb {DC}}^{4 \times 4}$ is given by
$$ U_{st} = \left[
\begin{array}{rrrr}
-0.5237           & -0.3744-0.0051\ii &  0.6728+0.0023\ii & -0.3646           \\
 0.4219+0.0715\ii & -0.7301+0.1798\ii &  0.1650+0.1569\ii &  0.4468-0.0091\ii \\
 0.0966-0.3988\ii & -0.4463+0.0870\ii & -0.4784-0.1274\ii & -0.5654+0.2451\ii \\
-0.0912+0.6050\ii & -0.1060+0.2764\ii & -0.3073+0.3948\ii & -0.3288-0.4238
\end{array} \right],
$$
$$ U_{\I} = \left[
\begin{array}{rrrr}
-0.3057-0.1480\ii &  0.0539-0.3954\ii &  0.1731-0.3048\ii &  0.7069-0.1562\ii \\
-0.3237-0.2247\ii & -0.0843+0.3301\ii & -0.3227+0.3169\ii &  0.0851+0.3018\ii \\
-0.0165-0.7156\ii &  0.1991+0.0450\ii &  0.7210-0.4880\ii & -0.2549-0.2455\ii \\
-0.0418-0.4878\ii & -0.3411-0.1952\ii &  0.1419+0.5425\ii & -0.7384+0.2460\ii
\end{array} \right].
$$
%$$ U_{st} = \left[
%\begin{array}{rrrr}
% 0.0825-0.0226\ii & -0.3657-0.0022\ii & -0.9096-0.0892\ii &  0.077 +0.1330\ii \\
% 0.1738-0.4804\ii & -0.3784-0.4644\ii &  0.0693+0.1850\ii & -0.4394-0.3849\ii \\
%-0.2786-0.0023\ii &  0.4642+0.3120\ii & -0.3124-0.1204\ii & -0.5446-0.4481\ii \\
%-0.7262+0.3561\ii & -0.1974-0.3945\ii & -0.0311+0.1121\ii &  0.2167-0.3013\ii
%\end{array} \right],
%$$
%$$ U_{I} = \left[
%\begin{array}{rrrr}
%-0.3172-0.2752\ii &  0.0209-0.4422\ii & 0.0453-0.1704\ii &  0.3245+0.2077\ii \\
% 0.3086+0.3894\ii & -0.641 +0.0663\ii & 0.3312+0.0611\ii &  0.0888+0.1911\ii \\
% 0.3161+0.2912\ii & -0.666 -0.1675\ii &-0.1604+0.1432\ii & -0.3546-0.5004\ii \\
%-0.2113+0.2491\ii & -0.0959-0.3485\ii & 0.4091-0.2536\ii & -0.6425+0.7239\ii
%\end{array} \right].
%$$
\end{Exa}
%\noindent
%\textbf{Example 8.3.}

%\noindent
%\textbf{Example 8.2.}

\begin{Exa}\label{Ex8.3}
We test proposed algorithms on some randomly generated data to illustrate the possible application to brain science. The dual complex matrix $A=A_{st} + A_\I\epsilon \in {\mathbb {DC}}^{10 \times 4}$ is given by
$$ A_{st} = \left[
\begin{array}{rrrr}
 0.6516-0.7586\ii &  0.7909-0.6119\ii & -0.9687-0.2484\ii & -0.8647+0.5022\ii \\
 0.2923+0.9563\ii & -0.9912-0.1326\ii &  0.0899+0.9960\ii & -0.8026+0.5965\ii \\
-0.5078-0.8615\ii & -0.5161+0.8565\ii &  0.7457+0.6663\ii & -0.7211-0.6928\ii \\
 0.7627+0.6467\ii & -0.2377-0.9713\ii &  0.8851-0.4655\ii & -0.9973+0.0734\ii \\
 0.6893+0.7245\ii & -0.5061-0.8625\ii & -0.9634+0.2681\ii & -0.5420-0.8404\ii \\
-0.9960-0.0896\ii &  0.5090+0.8608\ii &  0.7250-0.6887\ii & -0.5013-0.8653\ii \\
-0.9767+0.2148\ii &  0.6086-0.7935\ii & -0.9909+0.1345\ii &  0.7155+0.6986\ii \\
-0.7858+0.6185\ii & -0.0043-1.0000\ii & -0.2393+0.9710\ii & -0.7417-0.6708\ii \\
-0.7479-0.6638\ii & -0.9462-0.3235\ii & -0.8282-0.5605\ii &  0.6520-0.7582\ii \\
-0.9988+0.0497\ii &  0.7748+0.6322\ii & -0.9035+0.4286\ii & -0.9510-0.3092\ii
\end{array} \right]
$$
and
$$ A_{\I} = \left[
\begin{array}{rrrr}
 0.7493-0.4794\ii &  0.8886-0.3328\ii & -0.8709+0.0308\ii & -0.7670+0.7814\ii \\
 0.6452+0.3523\ii & -0.6383-0.7366\ii &  0.4428+0.3919\ii & -0.4497-0.0075\ii \\
-0.2580-0.8536\ii & -0.2663+0.8644\ii &  0.9955+0.6742\ii & -0.4713-0.6850\ii \\
 0.6595+0.8259\ii & -0.3409-0.7922\ii &  0.7819-0.2863\ii & -1.1005+0.2526\ii \\
 1.0198+0.9021\ii & -0.1756-0.6849\ii & -0.6328+0.4457\ii & -0.2114-0.6628\ii \\
-0.9302+0.1061\ii &  0.5748+1.0565\ii &  0.7908-0.4930\ii & -0.4355-0.6696\ii \\
-0.8158+0.1512\ii &  0.7694-0.8571\ii & -0.8300+0.0709\ii &  0.8764+0.6350\ii \\
-0.3431+0.6388\ii &  0.4385-0.9797\ii &  0.2035+0.9913\ii & -0.2989-0.6504\ii \\
-0.2804-0.0873\ii & -0.4786+0.2530\ii & -0.3606+0.0160\ii &  1.1196-0.1817\ii \\
-0.4792-0.1506\ii &  1.2944+0.4319\ii & -0.3839+0.2283\ii & -0.4314-0.5095\ii
\end{array} \right].
$$
Here the phase matrix $A_{st} \in {\mathbb {C}}^{10 \times 4}$ is randomly generated such that the modulus of its elements are equal to 1, and the relative phase matrix $A_\I \in {\mathbb {C}}^{10 \times 4}$ is calculated by subtracting from each element of $A_{st}$ the mean of its row. See \cite{ATWBG06}.  By using Algorithm \ref{Alg1}, all the eigenvalues of dual complex matrix $B=A^*A \in {\mathbb {DC}}^{4 \times 4}$ are positive dual numbers given by
$$
\Sigma_B = {\rm diag}\left( 16.3352+27.3465\epsilon, 12.836+22.9941\epsilon, 7.3681+9.9258\epsilon, 3.4607+4.1092\epsilon \right).
$$
The corresponding unitary matrix $U=U_{st} + U_\I \epsilon \in {\mathbb {DC}}^{4 \times 4}$ is give by
$$ U_{st} = \left[
\begin{array}{rrrr}
 0.6072           & -0.3081           &  0.4805           & -0.5527           \\
-0.4992-0.0776\ii & -0.2581-0.5157\ii & -0.1719+0.0728\ii & -0.5541+0.2654\ii \\
 0.4858-0.1602\ii &  0.2758-0.3732\ii & -0.2524+0.4806\ii &  0.1606+0.4500\ii \\
-0.0270-0.3371\ii & -0.2539+0.5411\ii & -0.3819+0.5419\ii & -0.2202-0.2008\ii
\end{array} \right]
$$
and
$$ U_{\I} = \left[
\begin{array}{rrrr}
-0.0144+0.5457\ii &  0.0625-0.7958\ii &  0.0924-0.0491\ii &  0.0297+0.3729\ii \\
-0.8851+0.3259\ii &  0.2346+0.7871\ii &  0.1595-0.4347\ii & -0.2008-0.0086\ii \\
-0.4833-0.0491\ii &  0.0896-0.3614\ii & -0.9630-0.0557\ii & -0.7932-0.0479\ii \\
 0.0467+0.5327\ii & -0.0852+0.5627\ii &  0.8119+0.2002\ii & -0.5946+0.3713\ii
\end{array} \right].
$$
In the principal component analysis, one can choose the first few columns of $U$ to generate principal components of the data. Besides, we also test the phase matrix $A_{st}$ with scale $20000 \times 500$, the unitary matrix $U$ can be computed within about 4.5 seconds.
\end{Exa}

\begin{Exa}\label{Ex8.4}
The truncated SVD is used to approximate the sample images: ``Peppers" and ``Lena". Both of them are gray-scale images of size $512\times 512$. For these images, we first use the 2-dimensional discrete Fourier transform to generate the complex matrices $A_{st}$ and $A_{\I}$ which make up the dual complex matrix $A=A_{st} + A_{\I}\epsilon \in {\mathbb{DC}^{512\times 512}}$. The low-rank approximation $A_k  \in {\mathbb{DC}^{512\times 512}}$ can be obtained by the truncated SVD with the first $k$ largest singular values. The original images are approximated by applying the 2-dimensional inverse discrete Fourier transform to the standard part and the infinitesimal part of $A_k$, respectively. Given the value of $k$, we have $\| A_k-A \|_F = \sqrt{\sum_{i=k+1}^{512} |\lambda_i|^2}$ and $ \| A\|_F =  \sqrt{\sum_{i=1}^{512} |\lambda_i|^2}$ where $\lambda_1 \geq \lambda_2 \geq  \ldots \geq \lambda_{512}$ are all the singualr values of $A$. Note that $\| A_k-A \|_F$, $\| A\|_F $ and $\| A_k-A \|_F/\|A \|_F$ are dual numbers since $\lambda_i$'s are dual numbers. The relative errors of approximation are reported in Table \ref{Tab1}. The approximated images with $k$-truncated SVD are also presented in Figure \ref{Fig1}.

\begin{table}[H]
\centering
\caption{Relative error of approxiamtion with $k$-truncated SVD} \label{Tab1} \vspace{3mm}
\begin{tabular}{|c|c|c|c|c|c|}\hline
$k$ & 5 & 15 & 25 & 35 & 45   \\\hline
%$\|A_k - A\|_F$ & $31144.327-407.4913\epsilon$ & 17685.5122-677.5773$\epsilon$ & 12871.9209+6.3281$\epsilon$ & 10251.2703-50.2026$\epsilon$ & 8511.2304+13.9977$\epsilon$ \\\hline
$\frac{\|A_k - A\|_F}{\|A\|_F}$ & 0.2546-0.2304$\epsilon$ & 0.1446-0.1345$\epsilon$ & 0.1052-0.0932$\epsilon$ & 0.0838-0.0752$\epsilon$ & 0.0696-0.0619$\epsilon$ \\\hline
\end{tabular}
\end{table}

\begin{figure}[H]
\centering
\includegraphics[width=0.9\textwidth]{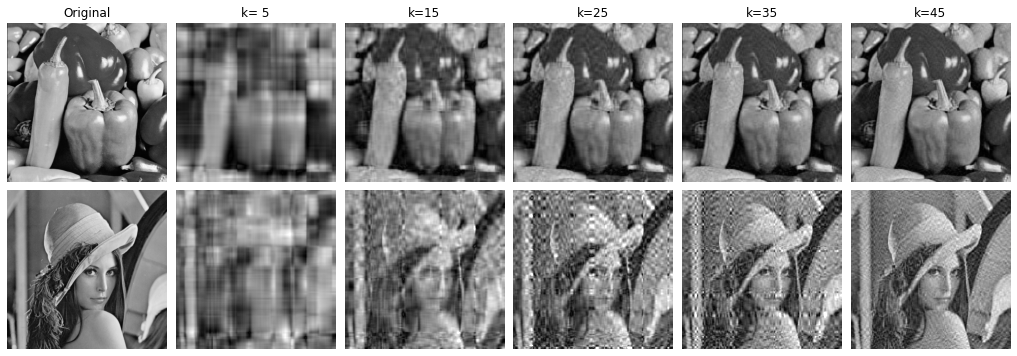}
\caption{Image approximation with $k$-truncated SVD of dual complex matrices}
\label{Fig1}
\end{figure}

\end{Exa}

\bigskip

%{\bf Acknowledgment}

%{\bf Data availability statement}    The datasets generated during and/or analysed during the current study are available from the corresponding author on reasonable request.form.

% \vspace{100pt}

\end{document}